\documentclass[12pt]{amsart}
\usepackage[english]{babel}
\usepackage[T1]{fontenc}
\usepackage[utf8]{inputenc}
\usepackage{indentfirst}
\usepackage{fancyhdr}
\usepackage{amsthm}
\usepackage{amsmath}
\usepackage{amssymb}
\usepackage{braket}
\usepackage{color,mathtools}
\usepackage{cite,enumitem,graphicx}
\usepackage[colorinlistoftodos]{todonotes}
\usepackage[colorlinks=true,linktocpage,
bookmarksnumbered,bookmarksopen,pdfpagelabels,]{hyperref}
\usepackage[hyperpageref]{backref}
\allowdisplaybreaks

\newtheorem{Th}{Theorem}[section]
\newtheorem{Prop}[Th]{Proposition}
\newtheorem{Lem}[Th]{Lemma}

\theoremstyle{definition}

\theoremstyle{remark}
\newtheorem{Rem}[Th]{Remark}

\newcommand{\R}{\mathbb{R}}
\newcommand{\cA}{\mathcal{A}}
\newcommand{\cB}{\mathcal{B}}
\newcommand{\cS}{\mathcal{S}}
\newcommand{\cD}{\mathcal{D}}
\newcommand{\cC}{\mathcal{C}}
\newcommand{\cF}{\mathcal{F}}
\newcommand{\cO}{\mathcal{O}}
\newcommand{\cG}{\mathcal{G}}
\newcommand{\cH}{\mathcal{H}}
\newcommand{\cR}{\mathcal{R}}
\newcommand{\dx}{\mathrm{d}x}
\newcommand{\dt}{\mathrm{d}t}
\newcommand{\dy}{\mathrm{d}y}
\newcommand{\dz}{\mathrm{d}z}
\newcommand{\D}{\mathrm{D}}

\numberwithin{equation}{section}

\title[A Gagliardo--Nirenberg inequality and mass-(sub)critical equations]{A generalisation of the Gagliardo--Nirenberg Inequality with applications to mass-critical and mass-subcritical elliptic equations}

\author[Bartosz Bieganowski]{Bartosz Bieganowski}
\address[Bartosz Bieganowski]{\newline \indent Faculty of Mathematics, Informatics and Mechanics, University of Warsaw, ul. Banacha 2, 02-097 Warsaw, Poland}
\email{\href{mailto:bartoszb@mimuw.edu.pl}{bartoszb@mimuw.edu.pl}}

\author[Jacopo Schino]{Jacopo Schino}
\address[Jacopo Schino]{\newline \indent Faculty of Mathematics, Informatics and Mechanics, University of Warsaw, ul. Banacha 2, 02-097 Warsaw, Poland}
\email{\href{mailto:j.schino2@uw.edu.pl}{j.schino2@uw.edu.pl}}

\subjclass{35J35, 35J75, 35J91, 35R11}

\keywords{Generalised Gagliardo--Nirenberg inequality, fractional and polyharmonic operators, Hardy-type potentials, normalised solutions}

\begin{document}

\begin{abstract}
Via a new inequality \`a la Gagliardo--Nirenberg, we prove the existence and nonexistence of solutions to
\begin{equation*}
\begin{cases}
(-\Delta)^s u + \frac{\mu}{|y|^{2s}} u + \lambda u = f(u), \quad \mathbb{R}^N \ni x = (y,z) \in \mathbb{R}^K \times \mathbb{R}^{N-K}, \\
\int_{\mathbb{R}^N} u^2 \, \dx = \rho
\end{cases}
\end{equation*}
in the mass-critical and mass-subcritical regimes, where $s>0$, $N \ge K \ge 2$, $\mu \in \R$ belongs to a specific range, $\rho>0$ is given a priori, and $\lambda \in \R$ is unknown. Additionally, we obtain similar results for the problem above with $\mu=0$ and $N \ge 1$ as well as a related curl-curl equation. Finally, we provide a thorough insight into the threshold for $\rho$ that divides the scenarios of negative and zero least energy.
\end{abstract}

\maketitle

\section{Introduction}\label{s:intro}

This paper concerns the equation
\begin{equation}\label{eq:}
(-\Delta)^s u + \frac{\mu}{|y|^{2s}} u + \lambda u = f(u), \quad \R^N \ni x = (y,z) \in \R^K \times \R^{N-K},
\end{equation}
paired with the constraint
\begin{equation}\label{eq:L2}
\int_{\R^N} u^2 \, \dx = \rho.
\end{equation}
Here, $s>0$, $N \ge K \ge 2$, $\mu \in \R$ belongs to a specific range, $\rho>0$ is given a priori, and $\lambda \in \R$ is unknown and arises as a Lagrange multiplier. We will also study the case $\mu=0$ and $N\ge1$.

If $f \colon \R \to \R$ is odd (hence extendable to the complex plane by $f(e^{\mathrm{i} \theta} r) = e^{\mathrm{i} \theta} f(r)$), \eqref{eq:} is obtained when seeking standing-wave solutions to the corresponding evolution equation
\begin{equation}\label{eq:time}
-\mathrm{i} \frac{\partial \Psi}{\partial t} + (-\Delta)^s \Psi + \frac{\mu}{|y|^{2s}} \Psi = f(\Psi), \quad \Psi \colon \R \times \R^N \to \mathbb{C},
\end{equation}
i.e., $\Psi(t,x) = e^{\mathrm{i} \lambda t} u(x)$.
Taking $s=1$ and $\mu=0$ in \eqref{eq:time}, we obtain the well-known (autonomous) \textit{semilinear Schr\"odinger equation} \cite{Cazenave}. When $0 < s < 1$, \eqref{eq:time} with $\mu=0$ was derived somewhat recently by N. Laskin \cite{LaskinE00,LaskinA,Laskin02} following an approach to quantum mechanics based on path integrals proposed by R. Feynman, who maintained that the wave function solving Schr\"odinger's equation should be given by the sum over all the possible histories of the system, i.e., by a heuristic integral over an infinity of quantum-mechanically possible trajectories \cite{Feynman}. Later on, the fractional semilinear Schr\"odinger equation with $s>1$ was used to describe anomalous diffusion in the Bloch--Torrey equation \cite{MABZ} and model acoustic wave propagation in heterogeneous attenuating media \cite{ZH}. When $s=2$, \eqref{eq:time} with $\mu=0$ was considered in \cite{IK,Turitsyn} to study the stability of solitons in magnetic materials once the effective quasi-particle mass becomes infinite. In general, higher-order powers are called for when the modelled phenomena become either more complex or encompass more complicated fundamental processes. For instance, polyharmonic equations were studied in the context of \textit{critical semilinear polyharmonic eigenvalue problems} \cite{Gazzola}, which are related to the \textit{Pucci--Serrin conjecture} \cite{PS}; more recently, they have been investigated in \cite{CCM}.

Concerning the singular term $\mu |y|^{-2s}$, when $N=K$ and $s=1$, this is the well-known \textit{Hardy potential}, closely related to the \textit{Hardy inequality} \cite{Hardy,RS}. As for the case $K < N$ (and $s=1$), one possible derivation is via the \textit{curl-curl equation}
\begin{equation}\label{eq:curl}
	\nabla \times \nabla \times \mathbf{U} = f(|\mathbf{U}|) \frac{\mathbf{U}}{|\mathbf{U}|}, \quad \mathbf{U} \colon \Omega \to \R^3,
\end{equation}
in domains $\Omega \subset \R^3$ invariant under the action of $\cO(2) \times \{1\}$, sometimes referred to as \textit{cylindrically symmetric}: here, the ansatz
\begin{equation}\label{eq:cyl}
	\mathbf{U}(x) = \frac{u(x)}{\sqrt{x_1^2 + x_2^2}}
	\begin{pmatrix}
		-x_2\\
		x_1\\
		0
	\end{pmatrix}
\end{equation}
for some $\cO(2) \times \{1\}$-invariant $u \colon \Omega \to \R$ leads to \eqref{eq:} (without the Lagrange multiplier $\lambda$) with $N=3$, $K=2$, and $\mu=1$ (more information on curl-curl problems in cylindrically-symmetric media is available in Section \ref{s:curl}). Different derivations are provided in \cite{BadBenRol,EL}.

The constraint \eqref{eq:L2} is motivated by the fact that the $L^2$ norm of solutions to \eqref{eq:time} is constant in time, not only for standing waves (which is obvious), but also, under mild conditions on $f$ (see (\ref{a:std}) below), for solutions such that $\Psi(t,\cdot)$ belongs to the right function space for every $t \in \R$
-- cf. $X^\mu$ below. This can be easily computed by multiplying \eqref{eq:time} by the complex conjugate of $\Psi$, integrating, and taking the imaginary part.

With the intention of studying \eqref{eq:}-\eqref{eq:L2} via variational methods, we introduce the space
\begin{equation*}
X^\mu \coloneqq \Set{u \in H^s(\R^N) | \left|\mu\right| \int_{\R^N} \frac{u^2}{|y|^{2s}} \, \dx < +\infty}
\end{equation*}
(note that $X^0 = H^s(\R^N)$), where
\begin{align*}
H^s(\R^N) & = \Set{u \in L^2(\R^N) | \int_{\R^N} |\xi|^{2s} |\cF(u)|^2 \, \mathrm{d}\xi < +\infty} \\
& = \Set{u \in L^2(\R^N) | \int_{\R^N} (1 + |\xi|^{2s}) |\cF(u)|^2 \, \mathrm{d}\xi < +\infty}
\end{align*}
and $\cF \colon L^2 (\R^N) \rightarrow L^2 (\R^N)$ is the Fourier transform. In view of the symmetry of \eqref{eq:}, we will introduce a symmetric subspace of $X^\mu$. Given $g \in \cO(K)$, we will denote
\begin{equation*}
\widetilde{g} \coloneqq
\begin{pmatrix}
g & 0\\
0 & I_{N-K}
\end{pmatrix}
\in \cO(N),
\end{equation*}
where $I_{N-K}$ is the $(N-K) \times (N-K)$ identity matrix, and
\begin{equation*}
\cG \coloneqq \Set{\widetilde g | g \in \cO(K)} \subset \cO(N).
\end{equation*}
It is understood that $\widetilde{g} = g$ and $\cG = \cO(N)$ when $K=N$. The subspace of $X^\mu$ consisting of $\cG$-invariant functions is defined as
\begin{equation*}
X_\cG^\mu \coloneqq \Set{u \in X^\mu | u(g\cdot) = u \text{ for all } g \in \cG}.
\end{equation*}
Observe that $\cG$ is compact and its action on $X^\mu$ is isometric; hence, we can exploit the principle of symmetric criticality \cite[Theorem on p. 22]{Palais}.

The assumptions about $f$ are as follows, where $F(t) \coloneqq \int_0^t f(r) \, \mathrm{d}r$,
\begin{equation*}
\alpha_N \coloneqq \frac{N(2\pi)^N}{\omega_{N-1}},
\end{equation*}
$\omega_{N-1}$ is the $(N-1)$-dimensional measure of the unit sphere $\mathbb{S}^{N-1}$, $2^* = +\infty$ if $N \le 2s$ and $2^* = \frac{2N}{N-2s}$ if $N > 2s$, and
\begin{equation*}
2_\# \coloneqq 2 + \frac{4s}{N}
\end{equation*}
is the so-called \textit{mass-critical} exponent.
\begin{enumerate}[label=(f\arabic{*}),ref=f\arabic{*}]
    \item \label{a:std} $f \in \cC(\R;\R)$ and
    \begin{itemize}
        \item if $N < 2s$, then $|f(t)| \lesssim |t|$ for all $t \in [-1,1]$;
        \item if $N = 2s$, then for every $q \ge 2$ and every $\alpha > \alpha_N$ there holds $|f(t)| \lesssim |t| + |t|^{q-1} (e^{\alpha t^2} - 1)$ for all $t \in \R$;
        \item if $N > 2s$, then $|f(t)| \lesssim |t| + |t|^{2^*-1}$ for all $t \in \R$.
    \end{itemize}
    \item \label{a:supq} $\displaystyle \lim_{t \to 0} \frac{F(t)}{t^2} = 0$.
    \item \label{a:subci} $\displaystyle \overline{\eta}_\infty \coloneqq
    \limsup_{|t| \to +\infty} \frac{F_+(t)}{|t|^{2_\#}} < +\infty$.
    \item \label{a:subcz} $\displaystyle \underline{\eta}_0 \coloneqq \liminf_{t \to 0} \frac{F_+(t)}{|t|^{2_\#}} > 0$.
\end{enumerate}

We introduce the functional $J \colon X^\mu \to \R$ defined by
\begin{equation*}
J(u) \coloneqq \frac12 \int_{\R^N} |\D^s u|^2 + \mu \frac{u^2}{|y|^{2s}} \, \dx - \int_{\R^N} F(u) \, \dx,
\end{equation*}
which is of class $\cC^1$ if \eqref{a:std} holds (see Lemma \ref{lem:L1} for the case $N=2s$), where
\begin{equation*}
\D^s u \coloneqq
\begin{cases}
\Delta^{s/2} u & \text{if } s \in \mathbb{N} \text{ is even}\\
\nabla \Delta^{(s-1)/2} u & \text{if } s \in \mathbb{N} \text{ is odd}\\
(-\Delta)^{s/2 - \lfloor s/2 \rfloor} (-\Delta)^{\lfloor s/2 \rfloor} u & \text{if } s \not\in \mathbb{N}
\end{cases}
\end{equation*}
and $\lfloor s/2 \rfloor$ is the integer part of $s/2$. We recall that, for $0 < s < 1$, the fractional Laplacian of $u$ is defined, e.g., via
\begin{equation*}
(-\Delta)^s u \coloneqq \cF^{-1}\bigl( |\cdot|^{2s} \cF (u) \bigr).
\end{equation*}

Concerning the range of $\mu$, we will prove in Lemma \ref{lem:Hardy} that if $K>2s$, then
\begin{equation}\label{eq:emb}
\int_{\R^N} \frac{u^2}{|y|^{2s}} \, \dx \le \left( \frac{\Gamma(\frac{K-2s}{4})}{2^s \Gamma(\frac{K+2s}{4})} \right)^2 \int_{\R^N} |\D^s u|^2 \, \dx \quad \text{for all } u \in \cD^{s,2}(\R^N),
\end{equation}
where $\Gamma$ is the Gamma function and
\begin{equation*}
\cD^{s,2}(\R^N) \coloneqq \Set{u \in L^{2^*}(\R^N) | \left| \D^s u \right| \in L^2(\R^N)}.
\end{equation*}
Consequently, we will always assume $\mu \ge 0$ when $K \le 2s$ and $\mu > -\left( 2^s \frac{\Gamma(\frac{K+2s}{4})}{\Gamma(\frac{K-2s}{4})} \right)^2$ when $K > 2s$ (in this case, $X^\mu = H^s(\R^N)$). To simplify notations, we define
\begin{equation*}
[u]_\mu^2 \coloneqq \int_{\R^N} |\D^s u|^2 + \mu \frac{u^2}{|y|^{2s}} \, \dx.
\end{equation*}

As stated earlier, we will use variational methods to find solutions to \eqref{eq:}-\eqref{eq:L2}; in particular, we will consider $J$ restricted to these subsets:
\begin{align*}
\cS_0(\rho) & \coloneqq \Set{u \in H^s(\R^N) | \int_{\R^N} u^2 \, \dx = \rho}, \quad \cS(\rho) \coloneqq \cS_0(\rho) \cap X_\cG^\mu, \\
\cD_0(\rho) & \coloneqq \Set{u \in H^s(\R^N) | \int_{\R^N} u^2 \, \dx \le \rho}, \quad \text{and} \quad \cD(\rho) \coloneqq \cD_0(\rho) \cap X_\cG^\mu.
\end{align*}

When $F$ grows as the power $2_\#$, at infinity or the origin, we need to restrict the range of $\rho>0$ to certain intervals. Specifically, we will consider the conditions
\begin{equation}\label{eq:etai}
2 C_{N,2_\#} \overline{\eta}_\infty \rho^{2s/N} < 1
\end{equation}
and
\begin{equation}\label{eq:etaz}
2 C_{N,2_\#} \underline{\eta}_0 \rho^{2s/N} > 1,
\end{equation}
where $C_{N,2_\#}$ is given in \eqref{eq:GNsym} for $p = 2_\#$. It is clear that \eqref{eq:etai} implies (\ref{a:subci}) and \eqref{eq:etaz} implies (\ref{a:subcz}).

Let us denote
\begin{equation}\label{eq:m}
m(\rho) \coloneqq \inf\Set{J(u) | u \in \cD(\rho)}
\end{equation}
and
\begin{equation}\label{eq:threshold}
\rho_* \coloneqq \inf \Set{\rho > 0 | m(\rho) < 0}.
\end{equation}
Of course, $\rho \mapsto m(\rho)$ is nonincreasing and, since $0 \in \cD(\rho)$, $m(\rho) \le 0$ for every $\rho > 0$. Moreover, it is clear that $m(\rho) = 0$ for all $\rho \in (0,\rho_*)$ and that $\rho_* = +\infty$ if $F \le 0$. Under the convention that $\frac10 = +\infty$, our first result about the existence of a solution to \eqref{eq:}-\eqref{eq:L2} is the following.

\begin{Th}\label{t:neg}
Assume $N \ge K \ge 2$, (\ref{a:std}), $F$ is positive somewhere. The following holds.
\begin{enumerate}
    \item $\rho_* \le \rho_F$, where    \begin{equation}\label{eq:rhoF}
    \rho_F \coloneqq \inf_{\sigma > 0} \left( \frac{1}{2 \sigma} \right)^{N / (2s)} \inf \Set{ [u]_\mu^{N/s} \left|u\right|_2^2 | u \in X_\cG^\mu \text{ and } \int_{\R^N} F(u) \, \dx \ge \sigma }.
    \end{equation}
    \item If (\ref{a:subcz}) holds, then $\rho_* \le (2 C_{N,2_\#} \underline{\eta}_0)^{-N/(2s)}$.
    \item Assume (\ref{a:supq}) holds, $\rho_*$ satisfies \eqref{eq:etai}, and $\rho \in \bigl( \rho_* , (2 C_{N,2_\#} \overline{\eta}_\infty)^{-N/(2s)} \bigr)$. If $(u_n)_n \subset \cD(\rho)$ and $\lim_n J(u_n) = m(\rho)$, then there exist $u \in \cS(\rho)$ and $\lambda > 0$ such that, up to subsequences and translations in the variable $z \in \R^{N-K}$, $u_n \to u$ in $X^\mu$, $J(u) = m(\rho) \in (-\infty,0)$, and $(\lambda,u)$ is a solution to \eqref{eq:}-\eqref{eq:L2}. Furthermore, every $\widetilde{u} \in \cD(\rho)$ such that $J(\widetilde{u}) = m(\rho)$ belongs to $\cS(\rho)$ and satisfies \eqref{eq:} for some $\widetilde{\lambda} > 0$.
\end{enumerate}
\end{Th}

The subset of $X_\cG^\mu$ in \eqref{eq:rhoF} is nonempty. Indeed, letting $R \gg 1$ and $t_* \in \R$ such that $F(t_*) > 0$, a function $u$ belonging to it can be built by taking $\phi \in \cC^{\infty}(\R,\R)$ such that $\phi \equiv 0$ in $(-\infty,-R-1] \cup [-1,1] \cup [R+1,+\infty)$, $\phi \equiv 1$ in $[-R,-2] \cup [2,R]$, $|\phi|_\infty \le 1$, and defining $u(x) = t_* \phi(|y|) \phi(|z|)$.

\begin{Rem}
If, in addition to the hypotheses of Theorem \ref{t:neg} (3), we assume $0 < s \le 1$, $K>2s$, and $-\left( 2^s \frac{\Gamma(\frac{K+2s}{4})}{\Gamma(\frac{K-2s}{4})} \right)^2 < \mu \le 0$, then there exists $v \in \cS(\rho)$ such that $J(v) = \min \Set{J(w) | w \in \cD_0(\rho)} = m(\rho)$. Indeed, taking $(u_n)_n \subset \cD_0(\rho)$ such that $\lim_n J(u_n) = \inf \Set{J(w) | w \in \cD_0(\rho)}$ and denoting $u_n^*$ the Schwarz rearrangement of $u_n$ with respect to $y \in \R^K$, we obtain $u_n^* \in \cD(\rho)$ and $J(u_n^*) \le J(u_n)$ for every $n$.
\end{Rem}

\begin{Rem}\label{rem:rho*}
If $\limsup_{t \to 0} F(t) |t|^{-2_\#} < +\infty$ and (\ref{a:std}) and (\ref{a:subci}) hold, then $\rho_* > 0$. Indeed, recalling that $m(\rho) \le 0$ because $0 \in \cD(\rho)$, it is obvious that $m(\rho) \ge 0$ if $F \le 0$. If, instead, $F$ is positive somewhere, then
\begin{equation*}
C_F \coloneqq \sup_{t \ne 0} \frac{F(t)}{|t|^{2_\#}} \in (0,+\infty).
\end{equation*}
Since, for every $u \in \cD(\rho)$,
\begin{equation*}
J(u) \ge \frac12 [u]_\mu^2 - C_F |u|_{2_\#}^{2_\#} \ge \left( \frac12 - C_F C_{N,2_\#} \rho^{2s/N} \right) [u]_\mu^2,
\end{equation*}
we have that $m(\rho) \ge 0$ for all $\rho \le (2 C_F C_{N,2_\#})^{-N/(2s)}$.
\end{Rem}

In the next theorem, recalling the convention that $\frac10 = +\infty$, we assume that
\begin{equation}\label{eq:zero_rho_star}
0 < \rho_* < (2 C_{N,2_\#} \overline{\eta}_\infty)^{-N/(2s)}.
\end{equation}
Notice that \eqref{eq:zero_rho_star} implies $F$ is positive somewhere (otherwise, as observed before, $\rho_* = +\infty$). We define
\begin{equation*}
\overline{\eta}_0 \coloneqq \limsup_{t \to 0} \frac{F_+(t)}{|t|^{2_\#}}
\end{equation*}
and introduce the condition
\begin{equation}\label{eq:two_etas}
2 C_{N,2_\#} \overline{\eta}_0 \rho^{2s/N} < 1
\end{equation}
as well as the set
\begin{equation}\label{eq:setB}
\cB(\rho) \coloneqq \Set{ u \in \cD(\rho) | [u]_\mu^2 > \gamma_* },
\end{equation}
where $\gamma_*>0$ is given in Lemma \ref{l:gamma} for $\rho = \rho_*$, and
\begin{equation}\label{eq:m_star}
m_*(\rho) \coloneqq \inf \Set{ J(u) | u \in \cB(\rho) }.
\end{equation}
In a way analogous to what happens for \eqref{eq:etai}, \eqref{eq:zero_rho_star} implies (\ref{a:subci}). Of course, (\ref{eq:etaz}) and \eqref{eq:two_etas} are incompatible.

\begin{Th}\label{t:pos}
Assume $N \ge K \ge 2$, (\ref{a:std})-(\ref{a:supq}), $\rho_*$ satisfies \eqref{eq:zero_rho_star}-\eqref{eq:two_etas}, and let $\overline{\rho}_*$ be given by Lemma \ref{lem:A}. If $\rho \in (\overline{\rho}_* , \rho_*]$, $(u_n)_n \subset \cB(\rho)$, and $\lim_n J(u_n) = m_*(\rho)$, then there exists $u \in \cB(\rho) \cap \cS(\rho)$ and $\lambda > 0$ such that, up to subsequences and translations in the variable $z \in \R^{N-K}$, $u_n \to u$ in $X^\mu$ and $(\lambda,u)$ is a solution to \eqref{eq:}-\eqref{eq:L2}. In particular, $m_*(\rho) = J(u) > 0$ if $\rho \in (\overline{\rho}_* , \rho_*)$ and $m_*(\rho_*) = J(u) = 0 = m(\rho_*)$.
\end{Th}

\begin{Rem}
In Theorem \ref{t:pos}, the assumption that $\rho_*$ satisfies \eqref{eq:zero_rho_star}-\eqref{eq:two_etas} could seem excessive for an outcome about $\rho \in (\overline{\rho}_* , \rho_*)$. Nonetheless, the proof of Theorem \ref{t:pos} is based on the preliminary result that a minimiser for $J|_{\cD(\rho_*)}$ in $\cS(\rho_*)$ exists, which is obtained, in turn, under such an assumption (cf. Lemma \ref{l:zero}). Additionally, Theorem \ref{t:noex} below provides sufficient conditions for minimisers of $J|_{\cD(\rho_*)}$ in $\cS(\rho_*)$ not to exist; as is demonstrated in Section \ref{s:app}, $\rho_*$ does not satisfy \eqref{eq:zero_rho_star}-\eqref{eq:two_etas} under such conditions.
\end{Rem}

We can now state our result about the nonexistence of global minimisers related to \eqref{eq:}-\eqref{eq:L2}. To the best of our knowledge, such a general result does not exist in the current literature.

\begin{Th}\label{t:noex}
Assume $N \ge K \ge 2$, (\ref{a:std}), and $0 < \rho_* < +\infty$. If $2_\# F(t) < f(t) t$ for all $t \ne 0$ or $f(t) t < 2_\# F(t)$ for all $t \ne 0$, then no nontrivial minimisers of $J|_{\cD(\rho_*)}$ exist.
\end{Th}

\begin{Rem}
As is evident from the proof of Theorem \ref{t:noex}, the condition $2_\# F(t) < f(t) t$ for all $t \ne 0$ (respectively, $f(t) t < 2_\# F(t)$) is used to get that $2_\# \int_{\R^N} F(u) \, \dx < \int_{\R^N} f(u)u \, \dx$ (respectively, $\int_{\R^N} f(u)u \, \dx < 2_\# \int_{\R^N} F(u) \, \dx$) for a certain $u \in X_\cG^\mu \setminus \{0\}$. When $s \ge 1$, this is obtained under the weaker condition that the strict inequality holds along two sequences $(t_n^+) \subset (0,+\infty)$ and $(t_n^-) \subset (-\infty,0)$ such that $\lim_n t_n^\pm = 0$, cf. \cite[Lemma 2.1]{BiegMed}; similar considerations hold for Theorem \ref{t:noex0} and Propositions \ref{pr:Ff}  and \ref{pr:fF}. On the other hand, when $0 < s < \frac12$, the proof of \cite[Lemma 2.1]{BiegMed} fails because $H^s(\R^N)$ contains characteristic functions -- see \cite[Lemma 6.1]{Khor_Rodrigo}.
\end{Rem}

We shall now state the analogues of Theorems \ref{t:neg}, \ref{t:pos}, and \ref{t:noex} for the case $\mu=0$ and $N\ge1$. Their proofs are omitted as similar to those of said theorems when $N > K$. We define
\begin{equation*}
m_0(\rho) \coloneqq \inf\Set{J(u) | u \in \cD_0(\rho)} \quad \text{and} \quad \rho_{*,0} \coloneqq \inf \Set{\rho > 0 | m_0(\rho) < 0}.
\end{equation*}

\begin{Th}\label{t:neg0}
Assume $N \ge 1$, $\mu=0$, (\ref{a:std}), and $F$ is positive somewhere. The following holds.
\begin{enumerate}
    \item We have
    \begin{equation*}
    \rho_{*,0} \le \inf_{\sigma > 0} \left( \frac{1}{2 \sigma} \right)^{N / (2s)} \inf \Set{ \left|\D^s u\right|^{N/s}_{2} \left|u\right|_2^2 | u \in H^s(\R^N) \text{ and } \int_{\R^N} F(u) \, \dx \ge \sigma }.
    \end{equation*}
    \item If (\ref{a:subcz}) holds, then $\rho_{*,0} \le (2 C_{N,2_\#}^0 \underline{\eta}_0)^{-N/(2s)}$, where $C_{N,2_\#}^0$ is given in \eqref{eq:GN} for $p = 2_\#$.
    \item Assume (\ref{a:supq}) holds, $2 C_{N,2_\#}^0 \overline{\eta}_\infty \rho_{*,0}^{2s/N} < 1$, and $\rho \in \bigl( \rho_{*,0} , (2 C_{N,2_\#}^0 \overline{\eta}_\infty)^{-N/(2s)} \bigr)$. If $(u_n)_n \subset \cD_0(\rho)$ and $\lim_n J(u_n) = m_0(\rho)$, then there exist $u \in \cS_0(\rho)$ and $\lambda > 0$ such that, up to subsequences and translations, $u_n \to u$ in $H^s(\R^N)$, $J(u) = \min_{\cD_0(\rho)} J < 0$, and $(\lambda,u)$ is a solution to \eqref{eq:}-\eqref{eq:L2}. Furthermore, every $\widetilde{u} \in \cD_0(\rho)$ such that $J(\widetilde{u}) = m_0(\rho)$ belongs to $\cS_0(\rho)$ and satisfies \eqref{eq:} for some $\widetilde{\lambda} > 0$.
\end{enumerate}
\end{Th}

Observe that, when $N=1$ and $s \in \mathbb{N}$, Theorem \ref{t:neg0} improves \cite[Theorem 1.4]{Schino_Smyrnalis}.
Next, we define
\begin{equation*}
\cB_0(\rho) \coloneqq \Set{ u \in \cD_0(\rho) | \left|\D^s u\right|^2_{2} > \gamma_{*,0} },
\end{equation*}
where $\gamma_{*,0}>0$ is given in Remark \ref{r:gamma} for $\rho = \rho_{*,0}$, and
\begin{equation*}
m_{*,0}(\rho) \coloneqq \inf \Set{ J(u) | u \in \cB_0(\rho) }.
\end{equation*}

\begin{Th}\label{t:pos0}
Assume $N\ge1$, $\mu=0$, (\ref{a:std})--(\ref{a:subci}), 
\begin{equation}\label{eq:rho0etas}
0 < \rho_{*,0} < (2 C_{N,2_\#}^0 \min \{ \overline{\eta}_\infty , \overline{\eta}_0 \})^{-N/(2s)},
\end{equation}
and let $\overline{\rho}_{*,0}$ be given by Remark \ref{rem:A}. If $\rho \in (\overline{\rho}_{*,0} , \rho_{*,0}]$, $(u_n)_n \subset \cB_0(\rho)$, and $\lim_n J(u_n) = m_{*,0}(\rho)$, then there exists $u \in \cB_0(\rho) \cap \cS_0(\rho)$ and $\lambda > 0$ such that, up to subsequences and translations, $u_n \to u$ in $H^s(\R^N)$ and $(\lambda,u)$ is a solution to \eqref{eq:}-\eqref{eq:L2}.
\end{Th}

Like for \eqref{eq:zero_rho_star}, \eqref{eq:rho0etas} implies $F$ is positive somewhere.

\begin{Th}\label{t:noex0}
Assume $N \ge 1$, $\mu=0$, (\ref{a:std}), and $0 < \rho_{*,0} < +\infty$. If $2_\# F(t) < f(t) t$ for all $t \ne 0$ or $f(t) t < 2_\# F(t)$ for all $t \ne 0$, then no nontrivial minimisers of $J|_{\cD_0(\rho_{*,0})}$ exist.
\end{Th}

\begin{Rem}
When $s=1$, arguing as in \cite[Proof of Theorem 1.4]{JL_min} or \cite[Proof of Lemma 3.7]{JL_M3AS}, we can prove that the global minimiser in Theorems \ref{t:neg} and \ref{t:neg0} and the local minimisers in Theorems \ref{t:pos} and \ref{t:pos0} are signed. Furthermore, if $N C_{N,2_\#} \overline{\eta} \rho^{2s/N} < 2s$, where $\overline{\eta} \coloneqq \limsup_{t \to 0} \bigl( f(t)t - 2F(t) \bigr)_+ |t|^{-2_\#}$, then we can argue as in \cite[Proof of Theorem 1.2]{JL_M3AS} and prove that, in the frameworks where every solution in $X_\cG^\mu$ to \eqref{eq:} satisfies the Poho\v{z}aev identity (respectively, every solution in $H^s(\R^N)$ satisfies the Poho\v{z}aev identity), the local minimiser given in Theorem \ref{t:pos} is a least-energy solution in $X_\cG^\mu$ (respectively, the local minimiser given in Theorem \ref{t:pos0} is a least-energy solution in $H^s(\R^N)$).
\end{Rem}

Except for part of Section \ref{s:app}, this paper is concerned with the so-called mass-subcritical and mass-critical regimes, i.e., when $J$ is bounded below over $\cS(\rho)$, respectively, for all or small values of $\rho$, depending on the growth of $F$ at infinity with respect to the power $2_\#$ -- cf. Lemma \ref{l:coer} below and \eqref{eq:etai}. This scenario was first investigated, for $s=1$ and $\mu=0$, in \cite{Lions84_2,Stuart82} and more recently in \cite{JL_M3AS,JL_min,JZZ_2024_JMPA,Schino,Shibata,Soave_JDE}. For $s>0$, the literature is more meagre. Relative to the mass-subcritical or mass-critical regimes, we mention \cite{CGT,JZh,KT,LZhZh} ($0<s<1$ and $N \ge 2$) and \cite{Schino_Smyrnalis} ($s \in \mathbb{N}$ and $N=1$); in all of them, $\mu=0$. Concerning the case $\mu \ne 0$, at least if we focus only on these two regimes, \cite{LW,SchinoPHD} seem to be the only pieces of work, and only for $s=\mu=1$ (we also point out \cite{DHZ}, which is concerned with systems).

We point out that $\cD(\rho)$ was first used in \cite{BiegMed} in the mass-supercritical regime, that is, when $J$ is unbounded below over $\cS(\rho)$ for all values of $\rho$. Later, it was adopted in different frameworks, such as the mass-subcritical and mass-critical regimes \cite{Schino} as well as nonlinearities that have a mass-subcritical growth at the origin and a mass-supercritical one at infinity \cite{BdS}; it was also utilised for singular polyharmonic operators, e.g., in \cite{BMS,BMiS}. In this paper, its use simplifies the proofs of Theorems \ref{t:neg} and \ref{t:pos} a little, as fewer intermediate results are needed.

The paper is structured as follows. A few preliminary results are given in Section \ref{s:prel}. In Section \ref{s:GN}, a new inequality \`a la Gagliardo--Nirenberg is stated and proved; the results contained therein are used throughout the paper. In Section \ref{s:norm}, the various outcomes about \eqref{eq:}-\eqref{eq:L2} are proved. Section \ref{s:app} contains scenarios where $\rho_*$ can be computed, shedding some light on the importance of \eqref{eq:zero_rho_star} and \eqref{eq:two_etas} and the link between Theorems \ref{t:pos} and \ref{t:noex}. In Section \ref{s:curl}, results about cylindrically symmetric curl-curl problems \eqref{eq:curl}-\eqref{eq:cyl} (and generalisations to higher dimensions) paired with constraints in the spirit of \eqref{eq:L2} are discussed. Appendix \ref{App} concerns a comparison between the two upper bounds for $\rho_*$ provided in Theorem \ref{t:neg}. To the best of our knowledge, Theorem \ref{t:noex} and the content of Section \ref{s:app} are completely new, at least in the degree of generality provided in this paper.

\subsection*{Notations} If $\mathfrak{f}$ is a real-valued function, then we set $\mathfrak{f}_\pm \coloneqq \max \{ \pm \mathfrak{f} , 0 \}$. When $N>K$, for $z \in \R^{N-K}$, we will denote $z' \coloneqq (0,z) \in \R^N$. For $R>0$ and $x \in \R^N$, $B_R(x)$ denotes the open ball with radius $R$ and centre $x$. For $u \in L^2(\R^N)$ and $t>0$, we denote $t \star u \coloneqq t^{N/2} u(t \cdot)$. When we write, e.g., that $u_n \to u$, we imply that the limit is taken as $n \to +\infty$. The symbol $\lesssim$ stands for the inequality $\le$ up to a positive multiplicative constant.

\section{Preliminaries}\label{s:prel}

The first part of this section contains useful preliminary results in various dimensions, albeit sometimes with restrictions. Then, in Subsection \ref{ss:NtwoS}, we will focus on the case $N=2s$.

\begin{Lem}\label{lem:Hardy}
	If $K>2s$, then \eqref{eq:emb} holds.
\end{Lem}
\begin{proof}
	If $N=K$, then the statement is proved in \cite[Theorem 2.5]{Herbst} (see also \cite[Subsection 1.1]{MY}). From now on, let us suppose $N>K$. We denote $\cF_y$ (respectively, $\cF_z$) the partial Fourier transform with respect to $y \in \R^K$ (respectively, $z \in \R^{N-K}$). For $\xi \in \R^N$, we will write $\xi = (\upsilon,\zeta) \in \R^K \times \R^{N-K}$.
    In virtue of the Fubini--Tonelli theorem, it is a matter of explicit computations that
    \begin{equation*}
    \cF(u)(\xi) = \cF_z\bigl(\cF_y(u)(\upsilon)\bigr)(\zeta) \quad \text{for all } u \in \cS(\R^N),
    \end{equation*}
    where $\cS(\R^N)$ is the Schwarz class. Now take $u \in \cD^{s,2}(\R^N)$. By density, we can assume $u \in \cS(\R^N)$. Then, using again the Fubini--Tonelli theorem and the result for $N=K$,
	\begin{align*}
		\int_{\R^N} \frac{u^2(x)}{|y|^{2s}} \, \dx & = \int_{\R^{N-K}} \int_{\R^K} \frac{u^2(y,z)}{|y|^{2s}} \, \dy \, \dz\\
		& \le \left( \frac{\Gamma\left(\frac{K-2s}{4}\right)}{2^s\Gamma\left(\frac{K+2s}{4}\right)} \right)^2 \int_{\R^{N-K}} \int_{\R^K} \left(\sum_{j=1}^K \xi_j^2\right)^{s} \left|\cF_y\bigl(u(\cdot,z)\bigr)(\upsilon)\right|^2 \, \mathrm{d}\upsilon \, \dz\\
		& = \left( \frac{\Gamma\left(\frac{K-2s}{4}\right)}{2^s\Gamma\left(\frac{K+2s}{4}\right)} \right)^2 \int_{\R^{N-K}} \int_{\R^K} \left(\sum_{j=1}^K \xi_j^2\right)^{s} \left|\cF_z\bigl(\cF_y(u)(\upsilon)\bigr)(\zeta)\right|^2 \, \mathrm{d}\upsilon \, \mathrm{d}\zeta\\
		& \le \left( \frac{\Gamma\left(\frac{K-2s}{4}\right)}{2^s\Gamma\left(\frac{K+2s}{4}\right)} \right)^2 \int_{\R^N} \left(\sum_{j=1}^N \xi_j^2\right)^{s} \left|\cF(u)(\xi)\right|^2 \, \mathrm{d}\xi\\
		& = \left( \frac{\Gamma\left(\frac{K-2s}{4}\right)}{2^s\Gamma\left(\frac{K+2s}{4}\right)} \right)^2 \int_{\R^N} |\xi|^{2s} |\cF(u)|^2 \, \mathrm{d}\xi.\qedhere
	\end{align*}
\end{proof}

We will need the following variants of Lions' lemma \cite{Lions84_2}.

\begin{Lem}\label{lem:Lions}
	Let $N \ge 1$ and $\Phi \in \cC(\R;\R)$ such that $\lim_{t \to 0} \Phi(t) t^{-2} = 0$. If $N \ge 2s$, assume, additionally, that
	\begin{equation*}
		\begin{cases}
			\displaystyle \lim_{|t| \to +\infty} \frac{\Phi(t)}{e^{\alpha t^2}} = 0 \text{ for all } \alpha > 0 & \text{if } N = 2s,\\
			\displaystyle \lim_{|t| \to +\infty} \frac{\Phi(t)}{|t|^{2^*}} = 0 & \text{if } N > 2s.
		\end{cases}
	\end{equation*}
	Let $(u_n)_n \subset X^\mu$ bounded.\\
	(i) If $\mu = 0$ and there exists $R>0$ such that
	\begin{equation*}
		\lim_n \sup_{x \in \R^N} \int_{B_R(x)} u_n^2 \, \dx = 0,
	\end{equation*}
	then $\lim_n |\Phi(u_n)|_1 = 0$.\\
	(ii) If $N > K \ge 2$, each $u_n$ is $\cG$-invariant, and for every $R > 0$ there holds
	\begin{equation}\label{eq:vanish}
		\lim_n \sup_{z \in \R^{N-K}} \int_{B_R(z')} u_n^2 \, \dx = 0,
	\end{equation}
	then $\lim_n |\Phi(u_n)|_1 = 0$.
\end{Lem}
\begin{proof}
	We prove only (i) when $N<2s$. The proof of (i) when $N \ge 2s$ follows verbatim that of \cite[Lemma 2.2]{BMS}, while (ii) can be proved as \cite[Corollary 3.2]{Mederski2020} (see also \cite[Remark 3.3]{Mederski2020}) having (i) at one's disposal. Fix $\varepsilon>0$. From the boundedness of $(u_n)_n$ and the embedding $H^s(\R^N) \hookrightarrow L^\infty(\R^N)$, there exists $c = c(\varepsilon) > 0$ such that
	\begin{equation*}
		|\Phi(t)| \le \varepsilon t^2 + c |t|^{4} \quad \text{for all } t \in \left[ -\sup_n |u_n|_\infty , \sup_n |u_n|_\infty \right],
	\end{equation*}
	therefore it suffices to prove that $u_n \to 0$ in $L^{4}(\R^N)$. From the interpolation inequality for Lebesgue spaces,
	\begin{equation*}
		|u_n|_{L^p(B_R(x))} \le |u_n|_{L^2(B_R(x))}^{1/2} |u_n|_{L^\infty(B_R(x))}^{1/2} \lesssim |u_n|_{L^2(B_R(x))}^{1/2} |u_n|_{H^s(B_R(x))}^{1/2}.
	\end{equation*}
Covering $\R^N$ with balls of radius $R$ such that each point is contained in at most $\mathcal{K}_N$ balls, where  $\mathcal{K}_N > 0$ depends only on $N$, we obtain
	\begin{equation*}
		|u_n|_p \lesssim \sup_k \|u_k\|^{1/2} \sup_{x \in \R^N} |u_n|_{L^2(B_R(x))}^{1/2},
	\end{equation*}
	whence the statement. Here, $\|\cdot\|$ is the norm in $X^\mu$.
\end{proof}

The next lemma is inspired by \cite[Lemma 5.5]{Hirata_Tanaka} (see also \cite[Proposition 4.2 (i)]{Gallo_Schino} for a proof that does not rely on the Nehari identity).
\begin{Lem}[Poho\v{z}aev identity]\label{lem:Poh}
	If (\ref{a:std}) holds and $u$ is a local minimiser of $J|_{\cD(\rho)}$, then there exists $\lambda \ge 0$ such that $(\lambda,u)$ is a solution to \eqref{eq:} and the Poho\v{z}aev identity holds:
	\begin{equation*}
		(N-2s) \int_{\R^N} |\mathrm{D}^s u|^2 + \frac{\mu}{|y|^{2s}} u^2 \, \dx = N \int_{\R^N} 2F(u) - \lambda u^2 \, \dx.
	\end{equation*}
\end{Lem}
\begin{proof}    
    We first verify that this function is differentiable in a neighbourhood of $t=1$. By a direct calculation, 
    $$ 
    J(t\star u) = \frac{t^{2s}}{2}[u]_\mu^2 - t^{-N}\int_{\mathbb R^N}F\left(t^{N/2}u\right)\,dx. 
    $$ 
    Thus, it suffices to justify the differentiation of 
    $$
    \Psi(t) := t^{-N}\int_{\mathbb R^N}F\left(t^{N/2}u\right)\,dx. 
    $$
    Fix $\delta > 0$, set $I:=[1-\delta,1+\delta]$, and define
    $$
    H(t,x):=t^{-N}F\left(t^{N/2} u(x)\right), \qquad (t,x)\in I\times\mathbb R^N. 
    $$
    For almost every $x\in\mathbb R^N$, the function $t\mapsto H(t,x)$ is continuously differentiable and 
    $$
    \partial_tH(t,x) = -Nt^{-N-1}F\bigl(t^{N/2} u(x)\bigr) + \frac{N}{2} t^{\frac{N}{2}-N-1}f\bigl(t^{N/2} u(x)\bigr)u(x). 
    $$
    We show that $\partial_tH(t,\cdot)$ is dominated by an $L^1$-function independent of $t\in I$.

    Suppose first that $N<2s$. Since $H^s(\mathbb R^N)\subset L^\infty(\mathbb R^N)$, we have $u\in L^\infty(\mathbb R^N)$. Let $R:=(1+\delta)^{N/2}|u|_\infty$. By the continuity of \(f\) and the first part of (\ref{a:std}), there exists $C_R>0$ such that 
    $$ 
    |f(\tau)|\leq C_R|\tau| \qquad\text{and}\qquad |F(\tau)|\leq C_R|\tau|^2 
    $$
    for every $|\tau|\leq R$. Consequently, uniformly for $t\in I$, 
    $$
    \left|\partial_tH(t,x)\right| \lesssim |u(x)|^2, 
    $$ 
    and the right-hand side belongs to $L^1(\mathbb R^N)$.

    Suppose next that $N>2s$. In this case, (\ref{a:std}) implies
    $$
    |f(\tau)| \lesssim |\tau|+|\tau|^{2^*-1}, \qquad |F(\tau)| \lesssim |\tau|^2+|\tau|^{2^*}
    $$
    for every $\tau \in \R$. Hence, it follows that
    $$
    \left|\partial_tH(t,x)\right| \lesssim |u(x)|^2+|u(x)|^{2^*},
    $$
    and the right-hand side is integrable.

    It remains to consider the critical case $N=2s$. Choose $\alpha > \alpha_N$ and apply (\ref{a:std}) with $q=2$. We obtain
    $$
    |f(\tau)| \lesssim |\tau|+|\tau|\left(e^{\alpha\tau^2}-1\right) .
    $$
    After integration, this also gives
    $$
    |F(\tau)| \lesssim |\tau|^2+|\tau|^2\left(e^{\alpha\tau^2}-1\right).
    $$
    Choose $\beta > \alpha(1+\delta)^N$. Then, for every $t \in I$ we have
    $$
    \left|\partial_tH(t,x)\right| \lesssim |u(x)|^2 + |u(x)|^2 \left( e^{\alpha (1+\delta)^N u(x)^2} - 1 \right) \lesssim |u(x)|^2+\left( e^{\beta u(x)^2}-1 \right).
    $$
    From Lemma \ref{lem:L1} we see that the right-hand side belongs to $L^1(\R^N)$. 
    
    Therefore, in all three cases, differentiation under the integral sign is justified by the dominated convergence theorem. Moreover, from the assumptions, there exists $\varepsilon \in (0,\delta)$ such that $t=1$ is a minimiser of $(1-\varepsilon,1+\varepsilon) \ni t \mapsto J(t \star u) \in \R$. This implies
	\begin{equation}\label{eq:D=0}
		0 = \frac{\mathrm{d}}{\mathrm{d} t}\Big|_{t=1} J(t \star u) = s [u]_\mu^2 + \frac{N}{2} \int_{\R^N} 2 F(u) - f(u)u \, \dx \eqqcolon M(u).
	\end{equation}
	From the Karush--Kuhn--Tucker theorem (see, e.g., \cite[Proposition A.1]{Mederski_Schino_2022}), there exists $\lambda \ge 0$ such that $(\lambda,u)$ is a solution to \eqref{eq:} and the Nehari identity holds:
	\begin{equation}\label{eq:Neh}
		\int_{\R^N} |\mathrm{D}^s u|^2 + \frac{\mu}{|y|^{2s}} u^2 + \lambda u^2 - f(u)u \, \dx = 0.
	\end{equation}
	Combining \eqref{eq:D=0} and \eqref{eq:Neh}, we obtain the statement.
\end{proof}

\subsection{Additional properties in the case $N=2s$}\label{ss:NtwoS}
We recall from \cite[Theorem 1.7]{Lam_Lu} that
\begin{equation}\label{eq:MT}
	\sup \Set{ \int_{\R^N} e^{\alpha_N u^2} - 1 \, \dx | u \in H^{N/2}(\R^N) \text{ and } |(\tau I - \Delta)^{N/4} u|_2 \le 1 } < +\infty,
\end{equation}
where $\tau$ is any positive number, $I \colon H^{N/2}(\R^N) \to H^{N/2}(\R^N)$ is the identity operator, and the constant $\alpha_N$ is sharp. Moreover, we need the following two properties.

\begin{Lem}\label{lem:L1}
	For every $\alpha>0$ and every $u \in H^{N/2}(\R^N)$, $e^{\alpha u^2} - 1 \in L^1(\R^N)$.
\end{Lem}
\begin{proof}
	For $\xi \in \R^N$, we set $\langle \xi \rangle \coloneqq \sqrt{1 + |\xi|^2}$. Since $\cF$ is an isometric isomorphism, we have
	\begin{equation*}
		H^{N/2}(\R^N) = \Set{ u \in L^2(\R^N) | \cF^{-1}(\langle \cdot \rangle^{N/2} \cF(u)) \in L^2(\R^N) }.
	\end{equation*}
	Denoting $v \coloneqq \cF^{-1}(\langle \cdot \rangle^{N/2} \cF(u)) \in L^2(\R^N)$, we obtain $u = \cF^{-1}(\langle \cdot \rangle^{-N/2}) \ast v$, where $\ast$ stands for the convolution. Consequently,
	\begin{equation*}
		H^{N/2}(\R^N) = \Set{ L \ast v | v \in L^2(\R^N) },
	\end{equation*}
	where
	\begin{equation*}
		L(x) \coloneqq \cF^{-1}(\langle \cdot \rangle^{-N/2}) (x) = \frac{1}{(4\pi)^{N/2} \Gamma(N/4)} \int_0^{+\infty} e^{-\pi |x|^2 / t - t / (4\pi)} t^{-N/4-1} \, \dt
	\end{equation*}
	is the Bessel potential.\\
	Now, take $u \in H^{N/2}(\R^N)$ and let $v \in L^2(\R^N)$ such that $u = L \ast v$. Denoting $\Omega \coloneqq \Set{ x \in \R^N | \left| u(x) \right| \ge 1 }$, we have
	\begin{equation*}
		\int_{\R^N} e^{\alpha u^2} - 1 \, \dx = \int_\Omega e^{\alpha u^2} - 1 \, \dx + \int_{\complement \Omega} e^{\alpha u^2} - 1 \, \dx.
	\end{equation*}
	Since there exists $C>0$ such that $e^{\alpha t^2}-1 \le C t^2$ for every $t \in [0,1]$, there holds
	\begin{equation*}
		\int_{\complement \Omega} e^{\alpha u^2} - 1 \, \dx \le C \int_{\R^N} u^2 \, \dx < +\infty.
	\end{equation*}
	Finally, using that $u = L \ast v$, we can follow the argument of \cite[Proof of Theorem 1.4]{Lam_Lu} and prove that $\displaystyle \int_\Omega e^{\alpha u^2} - 1 \, \dx < +\infty$.
\end{proof}

\begin{Lem}\label{lem:CLZ}
	For every $\alpha \in (0,\alpha_N)$,
	\begin{equation*}
		\sup \Set{ \frac{1}{\left|u\right|_2^2} \int_{\R^N} e^{\alpha u^2} - 1 \, \dx | u \in H^{N/2}(\R^N) \setminus \{0\} \text{ and } |\D^{N/2} u|_2 \le 1 } < +\infty.
	\end{equation*}
\end{Lem}
\begin{proof}
	Let $\alpha \in (0,\alpha_N)$. Using an internal scaling argument, it suffices to prove
	\begin{equation*}
		\sup \Set{ \int_{\R^N} e^{\alpha u^2} - 1 \, \dx | u \in H^{N/2}(\R^N) \text{ and } \max\{ |u|_2 , |\D^{N/2} u|_2 \} \le 1 } < +\infty.
	\end{equation*}
	Additionally, in view of \cite[Theorem 1.7]{Chen_Lu_Zhang}, we can assume that $N/2 \not \in \mathbb{N}$ and write $N/2 = m + 1/2$, $m \in \mathbb{N} \cup \{0\}$. If $u \in H^{m+1/2}(\R^N)$ is as above and $\tau \in (0,1)$, then
	\begin{align*}
		|(\tau I & - \Delta)^{(m+1/2)/2} u|_2^2
		= \int_{\R^N} (\tau + |\xi|^2)^{m+1/2} |\cF(u)|^2 \, \mathrm{d}\xi \\
		& \le \int_{\R^N} (\tau^{1/2} + |\xi|) (\tau + |\xi|^2)^m |\cF(u)|^2 \, \mathrm{d}\xi \\
		& = \sum_{j=0}^m
		{m \choose j}
		\left( \tau^{j+1/2} \int_{\R^N} |\xi|^{2(m-j)} |\cF(u)|^2 \, \mathrm{d}\xi + \tau^j \int_{\R^N} |\xi|^{2(m-j+1/2)} |\cF(u)|^2 \, \mathrm{d}\xi \right) \\
		& = \sum_{i=0}^{2m+1}
		{ m \choose  \lfloor i/2 \rfloor }
		\tau^{i/2} |\D^{m - i/2 + 1/2} u|_2^2 \\
		& \le |\D^{m + 1/2} u|_2^2 + \sum_{i=1}^{2m+1}
		{ m \choose  \lfloor i/2 \rfloor }
		\tau^{i/2} |u|_2^{2(1-\frac{2m-i+1}{2m+1})} |\D^{m+1/2} u|_2^{2\frac{2m-i+1}{2m+1}} \\
		& \le 1 + \tau^{1/2} \sum_{i=1}^{2m+1}
		{ m \choose  \lfloor i/2 \rfloor }.
	\end{align*}
	Therefore, we can take $\tau$ so small that
	\begin{equation*}
		|(\tau I - \Delta)^{(m+1/2)/2} u|_2^2 \le \frac{\alpha_N}{\alpha}
	\end{equation*}
	and have, from \eqref{eq:MT},
	\begin{equation*}
		\int_{\R^N} e^{\alpha u^2} - 1 \, \dx
		\le \int_{\R^N} e^{\alpha_N u^2 / |(\tau I - \Delta)^{(m+1/2)/2} u|_2^2} - 1 \, \dx
		\le C,
	\end{equation*}
	where $C>0$ does not depend on $u$.
\end{proof}

\section{A Gagliardo--Nirenberg-type inequality}\label{s:GN}

In this section, we fix $p \in (2,2^*)$. For every $u \in X_\cG^\mu \setminus \{0\}$, we define the quotient
\begin{equation*}
\cR(u) \coloneqq \frac{[u]_\mu^{p \delta_p} |u|_2^{p(1-\delta_p)}}{|u|_p^p} > 0,
\end{equation*}
where $\delta_p = \delta_p(N,s) = \frac{N}{s}(\frac12 - \frac1p) \in (0,1)$. It is readily seen that $\cR$ does not depend on internal or external scaling, that is, for all $t,r > 0$ and $u \in X_\cG^\mu \setminus \{0\}$, there holds $\cR(t u(r\cdot)) = \cR(u)$. As a consequence,
\begin{equation*}
\iota \coloneqq \inf \Set{\cR(u) | u \in X_\cG^\mu \setminus \{0\}} = \inf \Set{\cR(u) | u \in X_\cG^\mu \text{ and } [u]_\mu = |u|_2 = 1 }.
\end{equation*}

\begin{Th}\label{t:GN}
The infimum $\iota$ defined above is attained at a function $w \in X_\cG^\mu$ such that
\begin{equation*}
(-\Delta)^s w + \frac{\mu}{|y|^{2s}} w + \frac{1-\delta_p}{\delta_p} w = |w|^{p-2} w
\end{equation*}
and
\begin{equation}\label{eq:char}
[w]_\mu^{p-2} = |w|_2^{p-2} = \frac{\iota}{\delta_p}.
\end{equation}
\end{Th}
\begin{proof}
Let $(u_n)_n \subset X_\cG^\mu$ such that $[u_n]_\mu = |u_n|_2 = 1$ and $\lim_n \cR(u_n) = \iota$.\\
\textbf{Case 1:} $N=K$. From the compact embedding $H^s_\textup{rad}(\R^N) \hookrightarrow L^p(\R^N)$, there exists $u_0 \in X_\cG^\mu$ such that, up to a subsequence, $u_n \rightharpoonup u_0$ in $X^\mu$ and $u_n \to u_0$ in $L^p(\R^N)$. Consequently,
\begin{equation*}
\iota = \lim_n \frac{1}{|u_n|_p^p} = \frac{1}{|u_0|_p^p} \ge \frac{[u_0]_\mu^{p \delta_p} |u_0|_2^{p(1-\delta_p)}}{|u_0|_p^p} \ge \iota
\end{equation*}
(it is clear from the first two equalities above that $u_0 \ne 0$), which yields $[u_0]_\mu = |u_0|_2 = 1$ and $\cR(u_0) = \iota$.\\
\textbf{Case 2:} $N > K$. If the sequence $(u_n)_n$ satisfies \eqref{eq:vanish}, then $\lim_n |u_n|_p = 0$, and so
\[
\iota = \lim_n \frac{1}{|u_n|_p^p} = +\infty,
\]
which is impossible. Consequently, from Lemma \ref{lem:Lions} (ii), there exists $u_0 \in X_\cG^\mu \setminus \{0\}$ such that, up to a subsequence and translations in the variable $z \in \R^{N-K}$, $u_n \rightharpoonup u_0$ in $X^\mu$. Let $t_0,r_0 > 0$ such that $[t_0 u_0(r_0\cdot)]_\mu = |t_0 u_0(r_0\cdot)|_2 = 1$, that is,
\[
t_0 = \left( [u_0]_\mu^{-N/s} |u_0|_2^{N/s-2} \right)^{1/2} \text{ and } r_0 = \left( [u_0]_\mu^{-1} |u_0|_2 \right)^{1/s}.
\]
Then, $|t_0 u_0(r_0 \cdot)|_p = [u_0]_\mu^{-\delta_p} |u_0|_2^{\delta_p-1} |u_0|_p \ge |u_0|_p$, which implies
\begin{equation*}
\cR(t_0 u_0(r_0 \cdot)) = \cR(u_0) = \frac{[u_0]_\mu^{p \delta_p} |u_0|_2^{p(1-\delta_p)}}{|u_0|_p^p} \le \frac{1}{|u_0|_p^p} \le \frac{1}{|t_0 u_0(r_0 \cdot)|_p^p} = \cR(t_0 u_0(r_0 \cdot)).
\end{equation*}
This yields $[u_0]_\mu = |u_0|_2 = 1$, and so, from the consequent strong convergence in $X^\mu$, $\cR(u_0) = \iota$.\\
Since $\cR(u_0) = \min \cR$ and $[u_0]_\mu = |u_0|_2 = 1$, $u_0$ solves the equation
\begin{equation*}
(-\Delta)^s u_0 + \frac{\mu}{|y|^{2s}} u_0 + \frac{1-\delta_p}{\delta_p} u_0 = \frac{\iota}{\delta_p} |u_0|^{p-2} u_0,
\end{equation*}
whence the statement letting $w \coloneqq (\iota/\delta_p)^{1/(p-2)} u_0$.
\end{proof}

\begin{Rem}
The proof of Theorem \ref{t:GN} shows that every minimising sequence for $\iota$ is relatively compact in $X_\cG^\mu$ when $N=K$ or relatively compact in $X_\cG^\mu$ up to translations in the variable $z \in \R^{N-K}$ when $N>K$.
\end{Rem}

Finally, arguing in a similar way and using Lemma \ref{lem:Lions} (i), we can prove the following result.

\begin{Th}
Let $N \ge 1$ and define the quantities
\begin{equation*}
\cR_0(u) \coloneqq \frac{|\D^s u|_2^{p \delta_p} |u|_2^{p(1-\delta_p)}}{|u|_p^p}, \quad u \in H^s(\R^N) \setminus \{0\},
\end{equation*}
and
\begin{equation*}
\iota_0 \coloneqq \inf \Set{\cR_0(u) | u \in H^s(\R^N) \setminus \{0\}}.
\end{equation*}
Then, every minimising sequence for $\iota_0$ is relatively compact in $H^s(\R^N)$ up to translations, and $\iota_0$ is attained at a function $u \in H^s(\R^N)$ such that
\begin{equation*}
(-\Delta)^s u + \frac{1-\delta_p}{\delta_p} u = |u|^{p-2} u
\end{equation*}
and
\begin{equation*}
|\D^s u|_2^{p-2} = |u|_2^{p-2} = \frac{\iota_0}{\delta_p}.
\end{equation*}
\end{Th}

Letting $C_{N,p} \coloneqq \iota^{-1}$, we obtain the inequality
\begin{equation}\label{eq:GNsym}
|u|_p^p \le C_{N,p} [u]_\mu^{p \delta_p} |u|_2^{p(1-\delta_p)} \quad \text{for all } u \in X_\cG^\mu,
\end{equation}
where $C_{N,p}$ is optimal. Likewise, letting $C_{N,p}^0 \coloneqq \iota_0^{-1}$, we obtain the inequality
\begin{equation}\label{eq:GN}
|u|_p^p \le C_{N,p}^0 |\D^s u|_2^{p \delta_p} |u|_2^{p(1-\delta_p)} \quad \text{for all } u \in H^s(\R^N),
\end{equation}
where $C_{N,p}^0$ is optimal.

\section{Solutions with a prescribed norm}\label{s:norm}

Throughout this section, unless explicitly stated otherwise, $N \ge K \ge 2$. We begin by proving, through a series of lemmas, Theorem \ref{t:neg}.

\begin{Lem}\label{l:coer}
If (\ref{a:std}) and \eqref{eq:etai} hold, then $J|_{\cD(\rho)}$ is bounded below and coercive.
\end{Lem}
\begin{proof}
From (\ref{a:std}) and \eqref{eq:etai}, for every $\varepsilon>0$ there exists $c_\varepsilon>0$ such that
\begin{equation*}
F(t) \le c_\varepsilon t^2 + (\overline{\eta}_\infty + \varepsilon) |t|^{2_\#}
\end{equation*}
for all $t \in \R$. Consequently, if $u \in \cD(\rho)$, then \eqref{eq:GNsym} with $p = 2_\#$ yields
\begin{align*}
J(u) & \ge \frac12 [u]_\mu^2 - c_\varepsilon |u|_2^2 - (\overline{\eta}_\infty + \varepsilon) |u|_{2_\#}^{2_\#}\\
& \ge \frac12 [u]_\mu^2 - c_\varepsilon |u|_2^2 - (\overline{\eta}_\infty + \varepsilon) C_{N,2_\#} |u|_2^{4s/N} [u]_\mu^2\\
& \ge \frac{1}{2} [u]_\mu^2 - c_\varepsilon \rho - (\overline{\eta}_\infty + \varepsilon) C_{N,2_\#} \rho^{2s/N} [u]_\mu^2,
\end{align*}
and it suffices to take $\varepsilon$ sufficiently small.
\end{proof}

\begin{Rem}\label{rem:bbd}
Under (\ref{a:std}) and (\ref{a:subci}), the strict inequality in \eqref{eq:etai} is somewhat sharp. Indeed, if $f(t) = |t|^{2_\#-2} t$ (whence $\overline{\eta}_\infty = 1/2_\#$), $\rho = \bigl( 2_\# / (2 C_{N,2_\#}) \bigr)^{N/(2s)}$, and $w \in \cS(\rho)$ is the function given in Theorem \ref{t:GN} (up to scaling), then $[t \star w]_\mu \to +\infty$ as $t \to +\infty$, but $J(t \star w) = 0$ for all $t>0$. However, if equality holds in \eqref{eq:etai} and $F(t) \le \overline{\eta}_\infty |t|^{2_\#}$ for $|t| \gg 1$, then the argument of the proof of Lemma \ref{l:coer} shows that $J|_{\cD(\rho)}$ is bounded below. To some extent, this is again sharp, as \cite[Theorem 1.1 ii) a)]{Soave_JDE} constitutes a counterexample.
\end{Rem}

\begin{Lem}\label{lem:sub+}
Let $\rho_1,\rho_2 > 0$. If (\ref{a:std}) holds, then
\begin{equation*}
m(\rho_1 + \rho_2) \le m(\rho_1) + m(\rho_2),
\end{equation*}
and the inequality is strict if 
at least one between $m(\rho_1)$ and $m(\rho_2)$ is attained at a nonzero function.
\end{Lem}
\begin{proof}
We can assume that $m(\rho_1 + \rho_2) > -\infty$, whence $m(\rho_1) > -\infty$ and $m(\rho_2) > -\infty$ as well. Fix $\varepsilon>0$ and, for $j = 1,2$, let $u_j \in \cD(\rho_j)$ such that $J(u_j) \le m(\rho_j) + \varepsilon$. Next, let $\sigma \ge 1$ and observe that $u_j(\sigma^{-1/N} \cdot) \in \cD(\sigma \rho_j)$, whence
\begin{align*}
m(\sigma \rho_j) & \le J\bigl( u_j(\sigma^{-1/N} \cdot) \bigr) = \sigma \left( \frac{1}{2 \sigma^{2s/N}}  [u_j]_\mu^2 - \int_{\R^N} F(u_j) \, \dx \right)\\
&\le \sigma J(u_j) \le \sigma \bigl( m(\rho_j) + \varepsilon \bigr).
\end{align*}
In addition, $m(\sigma \rho_j) < \sigma m(\rho_j)$ if $\sigma > 1$ and $u_j \ne 0$.\\
If $\rho_1 \ge \rho_2$, then
\begin{align*}
m(\rho_1 + \rho_2) & \le \frac{\rho_1 + \rho_2}{\rho_1} \bigl( m(\rho_1) + \varepsilon \bigr) = m(\rho_1) + \frac{\rho_2}{\rho_1} m(\rho_1) + \frac{\rho_1 + \rho_2}{\rho_1} \varepsilon\\
& \le m(\rho_1) + m(\rho_2) + \frac{2 \rho_1 + \rho_2}{\rho_1} \varepsilon,
\end{align*}
and the first inequality is strict if $u_1 \ne 0$.\\
If $\rho_2 \ge \rho_1$, then
\begin{align*}
m(\rho_1 + \rho_2) & \le \frac{\rho_1 + \rho_2}{\rho_2} \bigl( m(\rho_2) + \varepsilon \bigr) = m(\rho_2) + \frac{\rho_1}{\rho_2} m(\rho_2) + \frac{\rho_1 + \rho_2}{\rho_2} \varepsilon\\
& \le m(\rho_1) + m(\rho_2) + \frac{\rho_1 + 2 \rho_2}{\rho_2} \varepsilon,
\end{align*}
and the first inequality is strict if $u_2 \ne 0$.\\
Since $\varepsilon$ is arbitrary, the first statement follows. The second can be proved exactly as in \cite[Proof of Lemma 2.5 (ii)]{Schino}.
\end{proof}

We recall from \eqref{eq:rhoF} the definition of $\rho_F$.

\begin{Lem}\label{lem:neg}
Assume that (\ref{a:std}) holds. Then, $m(\rho) < 0$ if one of the following scenarios occurs:
\begin{itemize}
    \item [(i)] $F$ is positive somewhere and $\rho > \rho_F$;
    \item [(ii)] \eqref{eq:etaz} holds.
\end{itemize}
\end{Lem}
\begin{proof}
Let us begin with case (i).
Fix $\sigma$ such that
\begin{equation*}
	\rho > \left( \frac{1}{2 \sigma} \right)^{N / (2s)} \inf \Set{ [u]_\mu^{N/s} \left|u\right|_2^2 | u \in X_\cG^\mu \text{ and } \int_{\R^N} F(u) \, \dx \ge \sigma }
\end{equation*}
and $v \in X_\cG^\mu$ such that $\int_{\R^N} F(v) \, \dx \ge \sigma$ and
\begin{equation*}
	\rho > \left( \frac{1}{2 \sigma} \right)^{N / (2s)} [v]_\mu^{N/s} \left|v\right|_2^2
\end{equation*}
Next, define $v_\rho \coloneqq v(|v|_2^{2/N} \rho^{-1/N} \cdot) \in \cS(\rho)$. Then,
\begin{align*}
	J(v_\rho) & = \frac12 [v]_\mu^2 |v|_2^{2 (2s/N - 1)} \rho^{1 - 2s/N} - \frac{\rho}{|v|_2^2} \int_{\R^N} F(v) \, \dx\\
	& \le \frac12 [v]_\mu^2 |v|_2^{2 (2s/N - 1)} \rho^{1 - 2s/N} - \frac{\sigma}{|v|_2^2} \rho\\
	& = \frac{\rho^{1 - 2s/N}}{|v|_2^2} \left( \frac12 [v]_\mu^2 |v|_2^{4s/N} - \sigma \rho^{2s/N} \right) < 0.
\end{align*}

Let us move to case (ii).
Let $u \in \cD(\rho) \setminus \{0\}$ and observe that
\begin{equation*}
	J(t \star u) = \frac{t^{2s}}{2} [u]_\mu^2 - t^{-N} \int_{\R^N} F(t^{N/2} u) \, \dx = t^{2s} \left( \frac{[u]_\mu^2}{2} - (t^{N/2})^{-2_\#} \int_{\R^N} F(t^{N/2} u) \, \dx \right).
\end{equation*}
We will prove preliminarily that
\begin{equation}\label{eq:noF-}
	\lim_{t \to 0} t^{-2_\#} \int_{\R^N} F_-(t u) \, \dx = 0,
\end{equation}
using that $F_- \equiv 0$ in a neighbourhood of the origin.\\
\textbf{Case 1:} $N > 2s$. From (\ref{a:std}),
\begin{equation*}
	F_-(t) \lesssim |t|^{2^*} \quad \text{for all } t \in \R,
\end{equation*}
whence
\begin{equation*}
	\lim_{t \to 0} t^{-2_\#} \int_{\R^N} F_-(t u) \, \dx \lesssim \lim_{t \to 0} t^{2^*-2_\#} |u|_{2^*}^{2^*} = 0.
\end{equation*}
\textbf{Case 2:} $N = 2s$ (hence $2_\# = 4$). Taking $\alpha = \alpha_N + 1$, we get from (\ref{a:std}) that
\begin{equation*}
	F_-(t) \lesssim |t|^5 (e^{\alpha t^2} - 1).
\end{equation*}
Using the Cauchy--Schwarz inequality and the property that $(e^a - 1)^b \le e^{ab} - 1$ for every $a \ge 0$ and $b \ge 1$, we obtain
\begin{align*}
	\int_{\R^N} F_-(tu) \, \dx & \lesssim |t|^5 \int_{\R^N} |u|^5 (e^{\alpha t^2 u^2} - 1) \, \dx\\
	& \le |t|^5 |u|_{10}^5 \sqrt{\int_{\R^N} e^{2 \alpha t^2 u^2} - 1 \, \dx},
\end{align*}
whence the claim.\\
\textbf{Case 3:} $N < 2s$. Since $X^\mu \hookrightarrow H^s(\R^N) \hookrightarrow L^\infty(\R^N)$, (\ref{a:std}) yields that
\begin{equation*}
	F_-(t) \lesssim |t|^{2_\#+1} \quad \text{for all } t \in [-|u|_\infty,|u|_\infty],
\end{equation*}
whence
\begin{equation*}
	\lim_{t \to 0} t^{-2_\#} \int_{\R^N} F_-(t u) \, \dx \lesssim \lim_{t \to 0} t |u|_{2_\#+1}^{2_\#+1} = 0.
\end{equation*}
Using \eqref{eq:noF-} and Fatou's lemma, we have
\begin{equation}\label{eq:referee}
	\limsup_{t \to 0} \frac{J(t \star u)}{t^{2s}} \le \frac{[u]_\mu^2}{2} - \liminf_{t \to 0^+} (t^{N/2})^{-2_\#} \int_{\R^N} F(t^{N/2} u) \, \dx \le \frac{[u]_\mu^2}{2} - \underline{\eta}_0 |u|_{2_\#}^{2_\#}.
\end{equation}
If $\underline{\eta}_0 = +\infty$, then $m(\rho) \le J(t \star u) < 0$ for sufficiently small $t>0$. Otherwise, we have again $m(\rho) < 0$ provided
\begin{equation}\label{eq:neg}
	\frac{[u]_\mu^2}{2} - \underline{\eta}_0 |u|_{2_\#}^{2_\#} < 0
\end{equation}
for some suitable $u$. Let $w \in X_\cG^\mu$ be the function given in Theorem \ref{t:GN} with $p = 2_\#$ (note that $\delta_{2_\#} = 2/2_\#$) and take $v = w(\tau \cdot)$ for some $\tau > 0$ to be determined. We want $|v|_2^2 \le \rho$ and \eqref{eq:neg} to be satisfied by $v$, i.e.,
\begin{equation}\label{eq:w}
	|w|_2^2 \le \rho \tau^N \quad \text{and} \quad \tau^{2s} [w]_\mu^2 < 2 \underline{\eta}_0 |w|_{2_\#}^{2_\#}.
\end{equation}
From \eqref{eq:char} and the minimality of $w$, it is a matter of straightforward computations to check that \eqref{eq:w} holds if and only if $\displaystyle \left( \frac{2_\#}{2 C_{N,2_\#}} \right)^{1/(2s)} \frac{1}{\rho^{1/N}} \le \tau < (2_\# \underline{\eta}_0)^{1/(2s)}$, and this interval is nonempty due to \eqref{eq:etaz}.
\end{proof}

\begin{Rem}\label{rem:neg}
If $\max \{ \liminf_{t \to 0^+} F(t) t^{-2_\#} , \liminf_{t \to 0^-} F(t) (-t)^{-2_\#} \} = +\infty$, then a slight modification of the proof of Lemma \ref{lem:neg} (ii) -- i.e., taking $u \ge 0$ or $u \le 0$ -- shows that $m(\rho) < 0$ for all $\rho > 0$.
\end{Rem}

\begin{Rem}
Having a Gagliardo--Nirenberg-type inequality that involves the singular integral $\mu \int_{\R^N} u^2 |y|^{-2s} \, \dx$ is crucial to obtain Lemma \ref{lem:neg} when $\mu > 0$. Indeed, if we used \eqref{eq:GN} instead of \eqref{eq:GNsym} in \eqref{eq:referee}, we would get
\begin{equation*}
\limsup_{t \to 0} \frac{J(t \star u)}{t^{2s}} \le \frac{|\mathrm{D}^s u|_2^2}{2} + \mu \int_{\R^N} \frac{u^2}{|y|^{2s}} \, \dx - \underline{\eta}_0 |u|_{2_\#}^{2_\#}.
\end{equation*}
Although a suitable $u$ such that $|\mathrm{D}^s u|_2^2 / 2 - \underline{\eta}_0 |u|_{2_\#}^{2_\#} < 0$ exists, we would still have to deal with the positive term $\mu \int_{\R^N} u^2 |y|^{-2s} \, \dx$.
\end{Rem}

\begin{Lem}\label{l:Lcpt}
Assume (\ref{a:std})-(\ref{a:supq}) and \eqref{eq:etai}. If $(u_n)_n \subset \cD(\rho)$ and $\lim_n J(u_n) = m(\rho)$, then there exist $u \in \cD(\rho)$ and $\lambda \ge 0$ such that, up to subsequences and translations, $u_n \rightharpoonup u$ in $L^2(\R^N)$, $[u_n - u]_\mu \to 0$, $J(u) = \min_{\cD(\rho)} J$, and $(\lambda,u)$ is a solution to \eqref{eq:}. If, moreover, $m(\rho) < 0$, then $\lambda > 0$, $u \in \cS(\rho)$, and $u_n \to u$ in $L^2(\R^N)$.
\end{Lem}
\begin{proof}
We shall argue up to subsequences without explicitly mentioning it. From Lemma \ref{l:coer}, $(u_n)_n$ is bounded in $X_\cG^\mu \hookrightarrow H^s(\R^N)$. Preliminarily, we will prove that, up to translations in the variable $z \in \R^{N-K}$ if $N>K$, $u_n \rightharpoonup u$ in $X_\cG^\mu$, $u_n \to u$ in $L^p(\R^N)$ for all $p \in (2,2^*)$, and $u \ne 0$ if $m(\rho) < 0$.\\
If $N=K$, then, from the compact embedding $X_\cG^\mu \hookrightarrow L^p(\R^N)$, there exists $u \in X_\cG^\mu$ such that $u_n \rightharpoonup u$ in $X_\cG^\mu$ and $u_n \to u$ in $L^p(\R^N)$ as $n \to +\infty$. Additionally, if $u=0$, then $\lim_n |F_+(u_n)|_1 = 0$ from (\ref{a:std})--(\ref{a:subci}), whence $m(\rho) = \lim_n J(u_n) \ge 0$, which is impossible if $m(\rho) < 0$.\\
Assume now that $N>K$. If \eqref{eq:vanish} holds, then $\lim_n |F_+(u_n)|_1 = 0$ from Lemma \ref{lem:Lions} and (\ref{a:std})--(\ref{a:subci}), and we reach a contradiction as before if $m(\rho) < 0$. Therefore, in this case, \eqref{eq:vanish} cannot hold, and, consequently, there exist $(z_n)_n \subset \R^{N-K}$ and $u \in X_\cG^\mu \setminus \{0\}$ such that $u_n(\cdot + z'_n) \rightharpoonup u$ in $X_\cG^\mu$. If, instead, $m(\rho) = 0$, we simply infer that $u_n \rightharpoonup u$ in $X_\cG^\mu$ as $n \to +\infty$ for some $u \in X_\cG^\mu$. As for the strong convergence in $L^p(\R^N)$, we distinguish the cases where \eqref{eq:vanish} holds (which implies $m(\rho)=0$) and does not hold (which implies $u\ne0$).\\
If \eqref{eq:vanish} holds, then, since $\lim_n |F_+(u_n)|_1 = 0$,
\begin{equation*}
0 = \lim_n J(u_n) = \lim_n \frac12 [u_n]_\mu^2 + \int_{\R^N} F_-(u_n) \, \dx \ge \frac12 \lim_n [u_n]_\mu^2 \ge 0,
\end{equation*}
i.e., $\lim_n [u_n]_\mu = 0$. Since $(u_n)_n$ is also bounded in $L^2(\R^N)$, from \eqref{eq:GNsym} we infer that $u_n \to 0$ in $L^p(\R^N)$.\\
If \eqref{eq:vanish} does not hold, then let $v_n \coloneqq u_n(\cdot + z_n') - u$. If $(v_n)_n$ satisfies \eqref{eq:vanish}, then Lemma \ref{lem:Lions} (ii) yields that $v_n \to 0$ in $L^p(\R^N)$. Consequently, assume by contradiction that
\begin{equation*}
\limsup_n \sup_{z \in \R^{N-K}} \int_{B_R(z')} v_n^2 \, \dx > 0
\end{equation*}
for some $R>0$. Then, as before, there exist $(\zeta_n)_n \subset \R^{N-K}$ and $v \in X_\cG^\mu \setminus \{0\}$ such that $w_n \coloneqq v_n(\cdot + \zeta_n') - v \rightharpoonup 0$ in $X_\cG^\mu$ and $w_n \to 0$ a.e. in $\R^N$. From the Brezis--Lieb lemma \cite{Brezis_Lieb},
\begin{equation*}
\lim_n J(u_n) - J(w_n) = J(u) + J(v) \quad \text{and} \quad \lim_n |u_n|_2^2 - |w_n|_2^2 = |u|_2^2 + |v|_2^2.
\end{equation*}
Denoting $\rho_u \coloneqq |u|_2^2 > 0$ and $\rho_v \coloneqq |v|_2^2 > 0$, there holds
\begin{equation*}
\rho - \rho_u - \rho_v \ge \liminf_n |u_n|_2^2 - \rho_u - \rho_v = \liminf_n |w_n|_2^2 \eqqcolon \rho' \ge 0.
\end{equation*}
If $\rho'>0$, then let $\widetilde{w}_n \coloneqq \sqrt{\rho'} \, |w_n|_2^{-1} w_n \in \cS(\rho')$. Arguing as in \cite[Lemma 2.4]{Shibata}, we obtain $\lim_n J(w_n) - J(\widetilde{w}_n) = 0$. From Lemma \ref{lem:sub+} and the monotonicity of $m$, there holds
\begin{equation*}
m(\rho) = \lim_n J(u_n) = \lim_n J(\widetilde{w}_n) + J(u) + J(v) \ge m(\rho') + m(\rho_u) + m(\rho_v) \ge m(\rho).
\end{equation*}
In particular, $m(\rho_u)$ and $m(\rho_v)$ are attained (at $u$ and $v$). This and Lemma \ref{lem:sub+} imply the contradiction $m(\rho') + m(\rho_u) + m(\rho_v) > m(\rho)$. If $\rho'=0$, then $w_n \to 0$ in $L^2(\R^N)$. Thus, from (\ref{a:std})-(\ref{a:supq}) and \eqref{eq:etai}, $F_+(w_n) \to 0$ in $L^1(\R^N)$, which implies $\liminf_n J(w_n) \ge 0$. Therefore,
\begin{equation*}
m(\rho) = \lim_n J(u_n) = \lim_n J(w_n) + J(u) + J(v) \ge m(\rho_u) + m(\rho_v) \ge m(\rho),
\end{equation*}
and we reach a contradiction as before.
\\
Finally, we prove the rest of the statement. We can assume that $u_n \to u$ a.e. in $\R^N$. Then, from (\ref{a:std})--(\ref{a:subci}) and the strong convergence in $L^p(\R^N)$,
\begin{equation*}
m(\rho) = \lim_n J(u_n) \ge J(u) \ge m(\rho),
\end{equation*}
which yields $J(u) = m(\rho)$ and $[u_n]_\mu^2 \to [u]_\mu^2$. In particular, from Lemma \ref{lem:Poh}, there exists $\lambda \ge 0$ such that $(\lambda,u)$ is a solution to \eqref{eq:} and the Poho\v{z}aev identity holds. If $m(\rho) < 0$ and $\lambda = 0$ (which holds if $u \in \cD(\rho) \setminus \cS(\rho)$), then $0 > J(u) = \frac{s}{N} [u]_\mu^2$, which is impossible. Hence, $\lambda > 0$ and $u \in \cS(\rho)$, which implies $\rho \ge \limsup_n |u_n|_2^2 \ge \liminf_n |u_n|_2^2 \ge |u|_2^2 = \rho$.
\end{proof}

\begin{Rem}\label{r:lambda}
The final part of the proof of Lemma \ref{l:Lcpt} demonstrates that $\lambda > 0$ even if $m(\rho) = 0$, provided that $u \ne 0$. It also proves that if $m(\rho) < 0$, then every $\widetilde{u} \in \cD(\rho)$ such that $J(\widetilde{u}) = m(\rho)$ belongs to $\cS(\rho)$ and satisfies \eqref{eq:} for some $\widetilde{\lambda} > 0$.
\end{Rem}

\begin{proof}[Proof of Theorem \ref{t:neg}]
It follows from Lemmas \ref{l:coer}, \ref{lem:neg}, \ref{l:Lcpt} and Remark \ref{r:lambda}.
\end{proof}

We move on, now, to Theorems \ref{t:pos} and \ref{t:noex}.

\begin{Lem}\label{l:noatt}
If (\ref{a:std}) and \eqref{eq:zero_rho_star} hold and $0 < \rho < \rho_*$, then $m(\rho) = 0$ is attained only at $0$.
\end{Lem}
\begin{proof}
Assume by contradiction that there exists $\rho \in (0,\rho_*)$ such that $m(\rho)$ is attained at a nonzero function. Let $\varepsilon > 0$ such that $\rho + \varepsilon < \rho_*$. From Lemma \ref{lem:sub+}, we reach the contradiction
\begin{equation*}
0 = m(\rho + \varepsilon) < m(\rho) + m(\varepsilon) = 0. \qedhere
\end{equation*}
\end{proof}

\begin{Lem}\label{l:gamma}
Assume (\ref{a:std}). There exists $\gamma>0$ such that, for all $u \in \cD(\rho)$ satisfying $[u]_\mu^2 \le 4 \gamma$, \eqref{eq:two_etas} implies
\begin{equation*}
J(u) \ge \frac12 \left( \frac12 - C_{N,2_\#} \overline{\eta}_0 \rho^{2s/N} \right) [u]_\mu^2.
\end{equation*}
\end{Lem}
\begin{proof}
Assume \eqref{eq:two_etas} and fix $\varepsilon > 0$ such that $2 C_{N,2_\#} (\overline{\eta}_0 + \varepsilon) \rho^{2s/N} < 1$. We shall consider the cases $N>2s$, $N=2s$, and $N<2s$ separately and denote by $C_\varepsilon$ a positive constant, depending on $\varepsilon$, that may vary after an inequality.\\
\textbf{Case 1:} $N>2s$. There exists $C_\varepsilon > 0$ such that
\begin{equation*}
F(t) \le (\overline{\eta}_0 + \varepsilon) |t|^{2_\#} + C_\varepsilon |t|^{2^*} \quad \text{for all } t \in \R.
\end{equation*}
Then, from \eqref{eq:GNsym} with $p = 2_\#$, for every $u \in \cD(\rho)$ such that $[u]_\mu^2 \le 4 \gamma$ there holds
\begin{align*}
J(u) & \ge \frac12 [u]_\mu^2 - (\overline{\eta}_0 + \varepsilon) |u|_{2_\#}^{2_\#} - C_\varepsilon |u|_{2^*}^{2^*} \\
& \ge \left( \frac12 - C_{N,2_\#} (\overline{\eta}_0 + \varepsilon) \rho^{2s/N} - C_\varepsilon \gamma^{{(2^*-2)/2}} \right) [u]_\mu^2,
\end{align*}
and it suffices to take $\varepsilon$ and $\gamma$ sufficiently small.\\
\textbf{Case 2:} $N=2s$ (hence $2_\# = 4$). There exists $C_\varepsilon > 0$ such that
\begin{equation*}
F(t) \le (\overline{\eta}_0 + \varepsilon) t^{4} + C_\varepsilon t^4 (e^{\alpha t^2} - 1) \quad \text{for all } t \in \R,
\end{equation*}
where, for simplicity, $\alpha = \alpha_N + 1$. We can assume that $8 \alpha \gamma < \alpha_N$. Then, from \eqref{eq:GNsym} with $p=4,8$ and Lemma \ref{lem:CLZ}, for every $u \in \cD(\rho)$ such that $[u]_\mu^2 \le 4 \gamma$ there holds
\begin{align*}
J(u) & \ge \frac12 [u]_\mu^2 - (\overline{\eta}_0 + \varepsilon) |u|_{4}^{4} - C_\varepsilon \int_{\R^N} u^4 (e^{\alpha u^2} - 1) \, \dx \\
& \ge \left( \frac12 - C_{N,4} (\overline{\eta}_0 + \varepsilon) \rho^{2s/N} \right) [u]_\mu^2 - C_\varepsilon |u|_8^4 \sqrt{\int_{\R^N} e^{2 \alpha u^2} - 1 \, \dx} \\
& \ge \left( \frac12 - C_{N,4} (\overline{\eta}_0 + \varepsilon) \rho^{2s/N} \right) [u]_\mu^2 - C_\varepsilon \sqrt{\rho} |u|_8^4 \\
& \ge \left( \frac12 - C_{N,4} (\overline{\eta}_0 + \varepsilon) \rho^{2s/N} - C_\varepsilon \gamma^{1/2} \rho^{2(1-\delta_8)+1/2} \right) [u]_\mu^2,
\end{align*}
and it suffices to take $\varepsilon$ and $\gamma$ sufficiently small.\\
\textbf{Case 3:} $N<2s$. There exists $C_\varepsilon > 0$ such that
\begin{equation*}
F(t) \le (\overline{\eta}_0 + \varepsilon) |t|^{2_\#} + C_\varepsilon |t|^{p} \quad \text{for all } t \in [-c,c],
\end{equation*}
where $c>0$ is such that $|u|_\infty \le c$ for all $u \in \cD(\rho)$ satisfying $[u]_\mu^2 \le 4 \gamma$ and, for simplicity, $p = 2_\#+1$. We can assume that $4\gamma \le 1$. Then, from \eqref{eq:GNsym} with $p = 2_\#$, for every $u \in \cD(\rho)$ such that $[u]_\mu^2 \le 4 \gamma$ there holds
\begin{align*}
J(u) & \ge \frac12 [u]_\mu^2 - (\overline{\eta}_0 + \varepsilon) |u|_{2_\#}^{2_\#} - C_\varepsilon |u|_{p}^{p} \\
& \ge \left( \frac12 - C_{N,2_\#} (\overline{\eta}_0 + \varepsilon) \rho^{2s/N} - C_\varepsilon \gamma^{(p\delta_p-2)/2} \rho^{p(1-\delta_p)/2} \right) [u]_\mu^2,
\end{align*}
and it suffices to take $\varepsilon$ and $\gamma$ sufficiently small.
\end{proof}

\begin{Rem}\label{r:gamma}
Under (\ref{a:std}), for a given $\rho>0$, one can prove the existence of $\gamma_0 > 0$ such that a result analogous to Lemma \ref{l:gamma} holds in the case $\mu=0$ and $N\ge1$.
\end{Rem}

\begin{Lem}\label{l:cont}
If (\ref{a:std}) and \eqref{eq:etai} hold, then $\bigl( 0 , (2 C_{N,2_\#} \overline{\eta}_\infty)^{-N/(2s)} \bigr) \ni \rho \mapsto m(\rho) \in (-\infty,0]$ is continuous. If, moreover, $\overline{\eta}_\infty > 0$ and $m\bigl( (2 C_{N,2_\#} \overline{\eta}_\infty)^{-N/(2s)} \bigr)$ is finite\footnote{A sufficient condition for this to hold is discussed in Remark \ref{rem:bbd}.}, then $m$ is left-continuous at $(2 C_{N,2_\#} \overline{\eta}_\infty)^{-N/(2s)}$.\\
In particular, if $m(\rho_*)$ is finite, then $m(\rho_*) = 0$.
\end{Lem}
\begin{proof}
Observe that $m(\rho) > -\infty$ for $0 < \rho < (2 C_{N,2_\#} \overline{\eta}_\infty)^{-N/(2s)}$ from Lemma \ref{l:coer}. Following \cite[Proof of Lemma 2.2 (v)]{JL_min}, we will prove that $\rho \mapsto m(\rho) / \rho$ is continuous. Given $u \in \cD(1)$, on $\bigl( 0 , (2 C_{N,2_\#} \overline{\eta}_\infty)^{-N/(2s)} \bigr)$ we define the function
\begin{equation*}
\rho \mapsto \Phi_u(\rho) \coloneqq \frac1\rho J\bigl( u(\cdot / \rho^{1/N}) \bigr) = \frac{1}{2 \rho^{2s/N}} [u]_\mu^2 - \int_{\R^N} F(u) \, \dx.
\end{equation*}
Note that $\Phi_u$ is concave (linear) in $\rho^{-2s/N}$. Since $m(\rho) / \rho = \inf_{u \in \cD(1)} \Phi_u(\rho)$ and $\rho \mapsto \rho^{-2s/N}$ is continuous, we obtain that $\rho \mapsto m(\rho) / \rho$ is continuous. Now, assume that $\overline{\eta}_\infty > 0$ and $m(\overline{\rho}) > -\infty$, where $\overline{\rho} \coloneqq (2 C_{N,2_\#} \overline{\eta}_\infty)^{-N/(2s)}$ for simplicity. From the monotonicity of $m$, it is clear that
\begin{equation*}
\liminf_{\rho \to \overline{\rho}^-} m(\rho) = \limsup_{\rho \to \overline{\rho}^-} m(\rho) \ge m(\overline{\rho}),
\end{equation*}
therefore it suffices to prove that, for every $u \in \cD(\overline{\rho})$,
\begin{equation*}
\lim_{\rho \to \overline{\rho}^-} m(\rho) \le J(u).
\end{equation*}
We fix $u \in \cD(\overline{\rho})$. If $u=0$, then the inequality above is obvious; hence, we can assume that $u \ne 0$. For $\varepsilon \in (0,|u|_2)$, we define $u_\varepsilon \coloneqq (|u|_2 - \varepsilon) |u|_2^{-1} u \in \cS(|u|_2 - \varepsilon) \subset \cD(\overline{\rho})$. Since $m$ is nonincreasing and $u_\varepsilon \to u$ in $X^\mu$ as $\varepsilon \to 0^+$, we obtain
\begin{equation*}
\lim_{\rho \to \overline{\rho}^-} m(\rho) \le \lim_{\varepsilon \to 0^+} J(u_\varepsilon) = J(u). \qedhere
\end{equation*}
\end{proof}

\begin{Lem}\label{l:zero}
If (\ref{a:std})-(\ref{a:supq}) hold and $\rho_*$ satisfies \eqref{eq:zero_rho_star} and \eqref{eq:two_etas}, then $m(\rho_*) = 0$ is attained at a function in $\cS(\rho_*)$ that is a solution to \eqref{eq:} for some $\lambda > 0$.
\end{Lem}
\begin{proof}
Let $(\rho_n)_n \subset \bigl( \rho_* , (2 C_{N,2_\#} \overline{\eta}_\infty)^{-N/(2s)} \bigr)$ be a decreasing sequence such that $\rho_n \to \rho_*$. For every $n$, since $- \infty < m(\rho_n) < 0$, we fix $u_n \in \cD(\rho_n) \subset \cD(\rho_1)$ such that
\begin{equation}\label{eq:mrhon}
m(\rho_n) \le J(u_n) < \frac12 m(\rho_n).
\end{equation}
From Lemma \ref{l:coer}, $(u_n)_n$ is bounded, while from Lemma \ref{l:gamma},
\begin{equation}\label{eq:[]>0}
\liminf_n [u_n]_\mu > 0.
\end{equation}
Assume $N>K$. If \eqref{eq:vanish} holds, then Lemma \ref{lem:Lions} together with (\ref{a:std})--(\ref{a:subci}) implies that $\lim_n |F_+(u_n)|_1 = 0$. This and \eqref{eq:mrhon} yield
\begin{equation*}
0 \ge \limsup_n J(u_n) \ge \liminf_n \frac12 [u_n]_\mu^2,
\end{equation*}
which contradicts \eqref{eq:[]>0}. Consequently, there exist $(z_n)_n \subset \R^{N-K}$ and $u \in X_\cG^\mu \setminus \{0\}$ such that $u_n(\cdot + z'_n) \rightharpoonup u$ in $X_\cG^\mu$ along a subsequence.\\
If $N=K$, using the compact embedding $X_\cG^\mu \hookrightarrow L^p(\R^N)$, we argue similarly and obtain $u \in X_\cG^\mu \setminus \{0\}$ such that $u_n \rightharpoonup u$ in $X_\cG^\mu$ along a subsequence (cf. the proof of Lemma \ref{l:Lcpt}).\\
If $N>K$, we can assume that $u_n(x + z'_n) \to u(x)$ for a.e. $x \in \R^N$; likewise, if $N=K$, we can assume that $u_n(x) \to u(x)$ for a.e. $x \in \R^N$. Let us denote $\rho_0 \coloneqq |u|_2^2 \in (0,\rho_*]$ and, for every $n$, $v_n \coloneqq u_n(\cdot + z_n') - u$. Since $\limsup_n |v_n|_2^2 \le \rho_* - \rho_0 < \rho_*$ and $m(\rho) = 0$ for every $\rho \in (0,\rho_*)$, we have that $J(v_n) \ge 0$ for every sufficiently large $n$. From Lemma \ref{l:cont}, \eqref{eq:mrhon}, and the Brezis--Lieb lemma \cite{Brezis_Lieb}, we obtain
\begin{equation*}
0 = \lim_n J(u_n) = J(u) + \lim_n J(v_n) \ge J(u) \ge 0,
\end{equation*}
i.e., $J(u) = m(\rho_0)$. Since, from Lemma \ref{l:noatt}, $ J(v) > m(\rho)$ for every $\rho \in (0,\rho_*)$ and every $v \in \cD(\rho) \setminus \{0\}$, we have $\rho_0 = \rho_*$, that is, $u \in \cS(\rho_*)$. Lemma \ref{lem:Poh} and Remark \ref{r:lambda} conclude the proof.
\end{proof}

\begin{Lem}\label{lem:A}
If (\ref{a:std})-(\ref{a:supq}) hold and $\rho_*$ satisfies \eqref{eq:zero_rho_star} and \eqref{eq:two_etas}, then there exists $\overline{\rho}_* \in (0,\rho_*)$ such that for all $\rho \in (\overline{\rho}_* , \rho_*]$ there holds
\begin{equation*}
m_*(\rho) < \min \left\{ \frac{s}{N} , \frac12 \left( \frac12 - C_{N,2_\#} \overline{\eta}_0 \rho^{2s/N} \right) \right\} \gamma_* \le \inf_{\cA(\rho)} J,
\end{equation*}
where
\begin{equation*}
\cA(\rho) \coloneqq \Set{ u \in \cB(\rho) | [u]_\mu^2 \le 4 \gamma_* }.
\end{equation*}
\end{Lem}
\begin{proof}
From Lemma \ref{l:zero}, there exists $u \in \cS(\rho_*)$ such that $J(u) = 0$. Additionally, from Lemma \ref{l:gamma},
\begin{equation*}
[u]_\mu^2 > 4 \gamma_* \quad \text{ and } \quad \inf_{\cA(\rho)} J \ge \frac12 \left( \frac12 - C_{N,2_\#} \overline{\eta}_0 \rho^{2s/N} \right) \gamma_*.
\end{equation*}
Then, there exists $\tau \in [\frac12,1)$ such that for all $t \in (\tau,1]$ we have
\begin{equation*}
[tu]_\mu^2 > \gamma_* \quad \text{ and } \quad J(tu) < \min \left\{ \frac{s}{N} , \frac12 \left( \frac12 - C_{N,2_\#} \overline{\eta}_0 \rho^{2s/N} \right) \right\} \gamma_*.
\end{equation*}
The proof is complete setting $\overline{\rho}_* \coloneqq \tau^2 \rho_*$.
\end{proof}

\begin{Rem}\label{rem:A}
Under (\ref{a:std})-(\ref{a:supq}) and \eqref{eq:rho0etas}, one can prove the existence of $\overline{\rho}_{*,0} \in (0,\rho_{*,0})$ such that a result analogous to Lemma \ref{lem:A} holds in the case $\mu=0$ and $N\ge1$.
\end{Rem}

\begin{Lem}\label{l:cpt}
Assume that (\ref{a:std})-(\ref{a:supq}) hold, $\rho_*$ satisfies \eqref{eq:zero_rho_star} and \eqref{eq:two_etas}, and let $\overline{\rho}_*$ be given by Lemma \ref{lem:A}. If $\rho \in (\overline{\rho}_* , \rho_*]$, $(u_n)_n \subset \cB(\rho)$, and $\lim_n J(u_n) = m_*(\rho)$, then there exists $u \in \cB(\rho) \cap \cS(\rho)$ and $\lambda > 0$ such that, up to subsequences and translations, $u_n \to u$ in $X^\mu$ and $(\lambda,u)$ is a solution to \eqref{eq:}-\eqref{eq:L2}.
\end{Lem}
\begin{proof}
From Lemma \ref{l:coer}, $(u_n)_n$ is bounded in $X_\cG^\mu \hookrightarrow H^s(\R^N)$.\\
Assume that $N>K$. If \eqref{eq:vanish} holds, then (\ref{a:std})--(\ref{a:subci}) and Lemma \ref{lem:Lions} yield $|F_+(u_n)|_1 \to 0$. Thus, from Lemma \ref{lem:A},
\begin{equation*}
\frac12 \left( \frac12 - C_{N,2_\#} \overline{\eta}_0 \rho^{2s/N} \right) \gamma_* \ge \lim_n J(u_n) \ge \frac12 \limsup_n [u_n]_\mu^2 \ge \frac12 \gamma_*,
\end{equation*}
which is impossible. Consequently, there exist $(z_n)_n \subset \R^{N-K}$ and $u \in \cD(\rho) \setminus \{0\}$ such that $u_n(\cdot + z_n') \rightharpoonup u$ in $X_\cG^\mu$ along a subsequence.\\
If $N=K$, using the compact embedding $X_\cG^\mu \hookrightarrow L^p(\R^N)$, we argue similarly and obtain $u \in X_\cG^\mu \setminus \{0\}$ such that $u_n \rightharpoonup u$ in $X_\cG^\mu$ along a subsequence (cf. the proof of Lemma \ref{l:Lcpt}). In what follows, to have uniform notations, we use the convention $z_n' = 0$ when $N=K$.\\
Let $v_n \coloneqq u_n(\cdot + z_n') - u$ and $\rho' \coloneqq |u|_2^2 \in (0,\rho]$. Observe that $[u_n]_\mu^2 > 4 \gamma_*$ for every sufficiently large $n$; otherwise, from Lemma \ref{lem:A}, $m_*(\rho) = \lim_n J(u_n) \ge \inf_{\cA(\rho)} J > m_*(\rho)$. Assume by contradiction that $u \not \in \cB(\rho')$, which implies
\begin{equation*}
\lim_n [v_n]_\mu^2 = \lim_n [u_n]_\mu^2 - [u]_\mu^2 \ge 3 \gamma_*.
\end{equation*}
In particular,
\begin{equation}\label{eq:vn}
[v_n]_\mu^2 \ge 2 \gamma_* \quad \text{if $n$ is sufficiently large.}
\end{equation}
We distinguish two cases.\\
\textbf{Case 1:} $\rho' = \rho$, i.e., $|v_n|_2 \to 0$. From (\ref{a:std}), (\ref{a:subci}), \eqref{eq:GNsym} with $p = 2_\#$, and \eqref{eq:vn} we have
\begin{equation*}
\lim_n J(v_n) \ge \frac12 \limsup_n [v_n]_\mu^2 \ge \gamma_*.
\end{equation*}
Then, using $J(u) \ge m(\rho) = 0$ and the Brezis--Lieb lemma \cite{Brezis_Lieb}, we obtain
\begin{equation*}
m_*(\rho) = \lim_n J(u_n) = J(u) + \lim_n J(v_n) \ge \gamma_*,
\end{equation*}
which contradicts Lemma \ref{lem:A}.\\
\textbf{Case 2:} $\rho' < \rho$. From \eqref{eq:vn} and the monotonicity of $m_*$ we have
\begin{equation*}
\lim_n J(v_n) \ge \limsup_n m_*(|v_n|_2^2) \ge m_*(\rho).
\end{equation*}
Since $m(\rho') = 0$ is attained only at $0$ (cf. Lemma \ref{l:noatt}), the Brezis--Lieb lemma \cite{Brezis_Lieb} yields the contradiction
\begin{equation*}
m_*(\rho) \ge J(u) + m_*(\rho) > m_*(\rho).
\end{equation*}
Having proved that $u \in \cB(\rho')$, we have $J(u) \ge m_*(\rho')$. Moreover, $v_n \in \cD(\rho)$ for every sufficiently large $n$ because $u \ne 0$, whence $\lim_n J(v_n) \ge m(\rho) = 0$. From the Brezis--Lieb lemma \cite{Brezis_Lieb} and the monotonicity of $m_*$ we infer
\begin{equation*}
m_*(\rho) = \lim_n J(u_n) = J(u) + \lim_n J(v_n) \ge  m_*(\rho') \ge m_*(\rho),
\end{equation*}
which implies
\begin{equation*}
J(u) = m_*(\rho).
\end{equation*}
From Lemma \ref{lem:Poh}, there exists $\lambda \ge 0$ such that $(\lambda,u)$ is a solution to \eqref{eq:} and satisfies the Poho\v{z}aev identity. In particular, if $\lambda = 0$ (which is the case if $\rho' < \rho$, cf. the proof of Lemma \ref{l:Lcpt}), then Lemma \ref{lem:A} yields
\begin{equation*}
\frac{s}{N} \gamma_* \ge m_*(\rho) = J(u) = \frac{s}{N} [u]_\mu^2 > \frac{s}{N} \gamma_*,
\end{equation*}
which is impossible. Consequently, $\rho' = \rho$ and $u_n(\cdot + z_n') \to u$ in $L^2(\R^N)$. Since $(u_n)_n$ is bounded in $X^\mu$, this implies that $u_n \to u$ in $L^{2_\#}(\R^N)$. Then, from (\ref{a:std}), (\ref{a:subci}), and \eqref{eq:GNsym} with $p = 2_\#$ we obtain
\begin{equation*}
m_*(\rho) = \lim_n \frac12 [u_n]_\mu^2 - \int_{\R^N} F(u_n) \, \dx \ge \frac12 [u]_\mu^2 - \int_{\R^N} F(u) \, \dx = m_*(\rho),
\end{equation*}
which allows us to conclude that $u_n(\cdot + z_n') \to u$ in $X^\mu$.
\end{proof}

\begin{proof}[Proof of Theorem \ref{t:pos}]
In view of Lemma \ref{l:cpt}, we are only left to prove that $m_*(\rho) > 0$ if $\rho \in (\overline{\rho}_* , \rho_*)$ and $m_*(\rho_*) = 0$. Clearly, $m_*(\rho') \ge m(\rho') = 0$ for all $\rho' \in (\overline{\rho}_* , \rho_*]$. If $\rho \in (\overline{\rho}_* , \rho_*)$, then Lemma \ref{l:noatt} yields that $m_*(\rho) > 0$. From Lemma \ref{l:gamma}, the global minimiser (say, $v$) given by Lemma \ref{l:zero} belongs to $\cB(\rho_*)$, hence $m_*(\rho_*) \le J(v) = m(\rho_*) = 0$.
\end{proof}
\begin{proof}[Proof of Theorem \ref{t:noex}]
The statement is obvious if $m(\rho_*) = -\infty$, so let us suppose the opposite. From Lemma \ref{l:cont}, $m(\rho_*) = 0$. If there exists $u \in \cD(\rho_*) \setminus \{0\}$ that minimises $J$ over $\cD(\rho_*)$, then $M(u) = 0$, where $M$ is defined in \eqref{eq:D=0}, and so, 
\begin{equation*}
0 = J(u) = J(u) - \frac{1}{2s} M(u) = \frac{N}{4s} \int_{\R^N} f(u) u - 2_\# F(u) \, \dx \ne 0. \qedhere
\end{equation*}
\end{proof}

\section{Counterexamples to \eqref{eq:zero_rho_star} or \eqref{eq:two_etas}}\label{s:app}

In this section, we provide an insight into the relation between Theorems \ref{t:pos} and \ref{t:noex}, together with two examples. Specifically, we present scenarios where $\rho_*$ can be computed, which, in turn, demonstrate why $\rho_*$ cannot satisfy \eqref{eq:zero_rho_star} (respectively, \eqref{eq:two_etas}) if $2_\# F(t) \le f(t) t$ (respectively, $f(t) t \le 2_\# F(t)$) for all $t \in \R$. We shall denote
\begin{equation*}
	\underline{\eta}_\infty \coloneqq \liminf_{|t| \to +\infty} \frac{F_+(t)}{|t|^{2_\#}}.
\end{equation*}

\begin{Prop}\label{pr:Ff}
	If (\ref{a:std}) holds, $F$ is positive somewhere, and $2_\# F(t) \le f(t) t$ for all $t \in \R$, then $\overline{\eta}_\infty > 0$, $\overline{\eta}_\infty \ge \overline{\eta}_0$, $\underline{\eta}_\infty \ge \underline{\eta}_0$,
	\begin{equation}\label{eq:Ff}
		(2 C_{N,2_\#} \underline{\eta}_\infty)^{-N/(2s)} \ge \rho_* \ge (2 C_{N,2_\#} \overline{\eta}_\infty)^{-N/(2s)},
	\end{equation}
	and $\rho_* = 0$ if $\overline{\eta}_\infty = +\infty$. If, additionally, $2_\# F(t) < f(t) t$ for all $t \ne 0$, then $\overline{\eta}_\infty > \overline{\eta}_0$ and $\underline{\eta}_\infty > \underline{\eta}_0$.
\end{Prop}
\begin{proof}
	Let $t_* \ne 0$ such that $F(t_*) > 0$. Observe that $t \mapsto F(t) |t|^{-2_\#}$ is nondecreasing on $(0,+\infty)$ and nonincreasing on $(-\infty,0)$, which implies $\overline{\eta}_\infty \ge \overline{\eta}_0$ and $\underline{\eta}_\infty \ge \underline{\eta}_0$. Moreover,
	\begin{equation}\label{eq:etaIsup}
		\overline{\eta}_\infty = \max \left\{ \lim_{t \to +\infty} \frac{F(t)}{t^{2_\#}} , \lim_{t \to -\infty} \frac{F(t)}{(-t)^{2_\#}} \right\} = \sup_{t \ne 0} \frac{F(t)}{|t|^{2_\#}} \ge \frac{F(t_*)}{|t_*|^{2_\#}} > 0.
	\end{equation}
	The second inequality in \eqref{eq:Ff} is obvious if $\overline{\eta}_\infty = +\infty$, while it follows from Remark \ref{rem:rho*} and \eqref{eq:etaIsup} otherwise. To prove the first, we will show that $m(\rho) = -\infty$ for every $\rho \in \bigl( (2 C_{N,2_\#} \underline{\eta}_\infty)^{-N/(2s)} , +\infty \bigr)$. Fix such a $\rho$ and assume that $\underline{\eta}_\infty = \lim_{t \to +\infty} F(t) t^{-2_\#}$ (the argument if $\underline{\eta}_\infty = \lim_{t \to -\infty} F(t) (-t)^{-2_\#}$ is analogous).\\
	\textbf{Case 1:} $\overline{\eta}_\infty = +\infty$ (i.e., $\lim_{t \to -\infty} F(t) (-t)^{-2_\#} = +\infty$). Fix $u \in \cS(\rho) \setminus \{0\}$ nonpositive. From the monotone convergence theorem,
	\begin{align*}
		J(t \star u) & = t^{2s} \left( \frac12 [u]_\mu^2 - t^{-N-2s} \int_{\R^N} F(t^{N/2} u) \, \dx \right) \\
		& = t^{2s} \left( \frac12 [u]_\mu^2 - \int_{\{u \ne 0\}} \frac{F(t^{N/2} u)}{(-t^{N/2} u)^{2_\#}} (-u)^{2_\#} \, \dx \right) \to -\infty \quad \text{as } t \to +\infty.
	\end{align*}
	\textbf{Case 2:} $\overline{\eta}_\infty < +\infty$. Let $w \in X_\cG^\mu$ be the function given by Theorem \ref{t:GN} for $p = 2_\#$. Up to rescaling, we can assume that $w \in \cS(\rho)$. We have
	\begin{equation}\label{eq:one}
		J(t \star w) = t^{2s} \left( \frac12 [w]_\mu^2 - \underline{\eta}_\infty |w|_{2_\#}^{2_\#} + \frac{1}{t^{N+2s}} \int_{\R^N} \underline{\eta}_\infty |t^{N/2} w|^{2_\#} - F(t^{N/2} w) \, \dx \right).
	\end{equation}
	Since $w$ satisfies \eqref{eq:GNsym} with equality and $p = 2_\#$,
	\begin{equation}\label{eq:two}
		\frac12 [w]_\mu^2 - \underline{\eta}_\infty |w|_{2_\#}^{2_\#} = [w]_\mu^2 \left( \frac12 - \underline{\eta}_\infty C_{N,2_\#} \rho^{2s/N} \right) < 0.
	\end{equation}
	Next, from (\ref{a:std}) and the fact that $\underline{\eta}_\infty = \sup_{t > 0} F(t) t^{-2_\#}$, there exists $C>0$ such that for all $t \ge 1$ there holds
	\begin{equation*}
		0 \le t^{-N-2s} \left( \underline{\eta}_\infty |t^{N/2} w_+|^{2_\#} - F(t^{N/2} w_+) \right) \le 2 \underline{\eta}_\infty |w_+|^{2_\#} + C w_+^2 \in L^1(\R^N).
	\end{equation*}
	Additionally, $t^{-N-2s} \bigl( \underline{\eta}_\infty |t^{N/2} w_+|^{2_\#} - F(t^{N/2} w_+) \bigr) \to 0$ a.e. as $t \to +\infty$, thus we can then use the dominated convergence theorem to infer that
	\begin{equation}\label{eq:three}
		\lim_{t \to +\infty} \frac{1}{t^{N+2s}} \int_{\R^N} \underline{\eta}_\infty |t^{N/2} w_+|^{2_\#} - F(t^{N/2} w_+) \, \dx = 0.
	\end{equation}
	Finally, in a way similar to Case 1, the monotone convergence theorem yields
	\begin{multline}\label{eq:four}
		\lim_{t \to +\infty} \frac{1}{t^{N+2s}} \int_{\R^N} \underline{\eta}_\infty |t^{N/2} w_-|^{2_\#} - F(-t^{N/2} w_-) \, \dx \\
		= \underline{\eta}_\infty |w_-|_{2_\#}^{2_\#} - \lim_{t \to +\infty} \int_{\R^N} \frac{F(-t^{N/2} w_-)}{t^{N+2s}} \, \dx
		= (\underline{\eta}_\infty - \overline{\eta}_\infty) |w_-|_{2_\#}^{2_\#} \le 0.
	\end{multline}
	From \eqref{eq:one}--\eqref{eq:four} conclude that $J(t \star w) \to -\infty$ as $t \to +\infty$.\\
	Lastly, if $\overline{\eta}_\infty = \overline{\eta}_0$ or $\underline{\eta}_\infty = \underline{\eta}_0$, then $t \mapsto F(t) |t|^{-2_\#}$ is constant on $(0,+\infty)$ or $(-\infty,0)$, which rules out that $2_\# F(t) < f(t) t$ for all $t \ne 0$.
\end{proof}

\begin{Rem}\label{r:Ff}
	Under the assumptions of Proposition \ref{pr:Ff}, $\rho_* = (2 C_{N,2_\#} \overline{\eta}_\infty)^{-N/(2s)}$ if $w \in X_\cG^\mu$ given by Theorem \ref{t:GN} for $p = 2_\#$ is signed\footnote{This is the case, e.g., if $0 < s \le 1$, as one can replace $w$ with $|w|$ or $-|w|$.}. Indeed, since $-w$ satisfies the same properties as $w$, we can assume that $w \le 0$. Moreover, we can assume that $\overline{\eta}_\infty < +\infty$. Replacing $\underline{\eta}_\infty$ with $\overline{\eta}_\infty$ in \eqref{eq:one}--\eqref{eq:four} and taking $\rho \in \bigl( (2 C_{N,2_\#} \overline{\eta}_\infty)^{-N/(2s)} , +\infty \bigr)$, \eqref{eq:two} is still satisfied, \eqref{eq:three} becomes trivial, and \eqref{eq:four} holds with equality.
\end{Rem}

\begin{Prop}\label{pr:fF}
	If (\ref{a:std}) holds, $F$ is positive somewhere, and $f(t) t \le 2_\# F(t)$ for all $t \in \R$, then $\overline{\eta}_0 > 0$, $\overline{\eta}_0 \ge \overline{\eta}_\infty$, $\underline{\eta}_0 \ge \underline{\eta}_\infty$, and
	\begin{equation}\label{eq:fF}
		\rho_* \ge (2 C_{N,2_\#} \overline{\eta}_0)^{-N/(2s)}.
	\end{equation}
	If, additionally, $f(t) t < 2_\# F(t)$ for all $t \ne 0$, then $\overline{\eta}_0 > \overline{\eta}_\infty$ and $\underline{\eta}_0 > \underline{\eta}_\infty$.
\end{Prop}
\begin{proof}
	Let $t_* \ne 0$ such that $F(t_*) > 0$. Observe that $t \mapsto F(t) |t|^{-2_\#}$ is nonincreasing on $(0,+\infty)$ and nondecreasing on $(-\infty,0)$, which implies $\overline{\eta}_0 \ge \overline{\eta}_\infty$ and $\underline{\eta}_0 \ge \underline{\eta}_\infty$. Moreover,	\begin{equation}\label{eq:eta0sup}
		\overline{\eta}_0 = \max \left\{ \lim_{t \to 0^+} \frac{F(t)}{t^{2_\#}} , \lim_{t \to 0^-} \frac{F(t)}{(-t)^{2_\#}} \right\} = \sup_{t \ne 0} \frac{F(t)}{|t|^{2_\#}} \ge \frac{F(t_*)}{|t_*|^{2_\#}} > 0.
	\end{equation}
	If $\overline{\eta}_0 = +\infty$, then \eqref{eq:fF} follows from Remark \ref{rem:neg} (and equality holds). To prove \eqref{eq:fF} when $\overline{\eta}_0 < +\infty$, we need to show that $J|_{\cD(\overline \rho)} \ge 0$, where $\overline{\rho} \coloneqq (2 C_{N,2_\#} \overline{\eta}_0)^{-N/(2s)}$. To this end, take $u \in \cD(\overline \rho)$. Then, from \eqref{eq:GNsym} with $p = 2_\#$ and \eqref{eq:eta0sup},
	\begin{align*}
		J(u) & = \frac12 [u]_\mu^2 - \overline{\eta}_0 |u|_{2_\#}^{2_\#} + \int_{\R^N} \overline{\eta}_0 |u|^{2_\#} - F(u) \, \dx \\
		& \ge \left( \frac12 - \overline{\eta}_0 C_{N,2_\#} \overline{\rho}^{2s/N} \right) [u]_\mu^2 = 0.
	\end{align*}
	Finally, if $\overline{\eta}_0 = \overline{\eta}_\infty$ or $\underline{\eta}_0 = \underline{\eta}_\infty$, then $t \mapsto F(t) |t|^{-2_\#}$ is constant on $(0,+\infty)$ or $(-\infty,0)$, which rules out that $f(t) t < 2_\# F(t)$ for all $t \ne 0$.
\end{proof}

\begin{Rem}\label{r:fF}
	Under the assumptions of Proposition \ref{pr:fF}, if $\lim_{t \to 0^+} F(t) t^{-2_\#} = \lim_{t \to 0^-} F(t) (-t)^{-2_\#}$, then Lemma \ref{lem:neg} (ii) and \eqref{eq:fF} yield $\rho_* = (2 C_{N,2_\#} \overline{\eta}_0)^{-N/(2s)} = (2 C_{N,2_\#} \underline{\eta}_0)^{-N/(2s)}$.
\end{Rem}

As examples of functions that satisfy the assumptions of Propositions \ref{pr:Ff} and \ref{pr:fF} -- in fact, of Proposition \ref{pr:Ff} and Remark \ref{r:fF} -- we mention, respectively,
\begin{equation*}
	f(t) = \min \{ |t|^{p-2} , |t|^{4s/N} \} t \quad \text{and} \quad f(t) = |t|^{4s/N}t - |t|^{p-2} t
\end{equation*}
with $2_\# < p < 2^*$. The latter was already mentioned in \cite[Remark 6.4]{JL_M3AS} in the special case $N=2$, $s=1$, and $p=6$.

\section{Curl-curl equations}\label{s:curl}
As outlined in Section \ref{s:intro}, a possible derivation of \eqref{eq:} (without $\lambda$) with $N=3$, $K=2$, and $s=\mu=1$ is the search for cylindrically symmetric solutions \eqref{eq:cyl} to curl-curl equations \eqref{eq:curl}. This ansatz was first adopted in \cite{AzzBenDApFor,BDPR}, where the authors worked on the vectorial equation \eqref{eq:curl} reducing the curl-curl operator to the negative vector Laplacian, and later in \cite{MederskiJFA,Zeng} (in $\R^3$ or 3-dimensional cylindrically symmetric bounded domains), where the authors focused on the corresponding equation in $u$. Similar symmetries were also considered in \cite{McLeodStuartTroy,Stuart90}, where cylindrically symmetric travelling-wave solutions were obtained.

Recall that, in $N=3$, the identity
$$
\nabla \times (\nabla \times \mathbf{U}) = \nabla (\operatorname{div} \mathbf{U}) - \Delta \mathbf{U}
$$
holds for every $\mathbf{U} \in C^2 (\R^3; \R^3)$. Since the right-hand side is well-defined for general dimensions $N \geq 3$, we can use it as the definition of $\nabla \times (\nabla \times \mathbf{U})$ for $\mathbf{U} \in C^2 (\R^N; \R^N)$. In the same way, we propose the weak formulation by setting
$$
\int_{\R^N} \langle \nabla \times \mathbf{U}, \nabla \times \mathbf{V} \rangle \, \dx \coloneqq \int_{\R^N} \langle \nabla \mathbf{U}, \nabla \mathbf{V} \rangle - (\operatorname{div} \mathbf{U}) (\operatorname{div} \mathbf{V}) \, \dx
$$
for $\mathbf{U}, \mathbf{V} \in H^1 (\R^N; \R^N)$.

Consider the following problem:
\begin{equation}\label{eq:curlcurl}
\left\{ \begin{array}{l}
     \nabla \times (\nabla \times \mathbf{U}) + \lambda \mathbf{U} = g(\mathbf{U}), \quad x \in \R^{N}, \\
     \displaystyle \int_{\R^N} | \mathbf{U} |^2 \, \dx = \rho,
\end{array} \right.
\end{equation}
where $\mathbf{U} \colon \R^N \rightarrow \R^N$ is an unknown vector field and $g \colon \R^N \rightarrow \R^N$ a nonlinear odd function. Looking for classical solutions of the form
\begin{equation}\label{eq:form}
\mathbf{U} (x) = \frac{u(x)}{r} \begin{pmatrix}
     -x_2 \\ x_1 \\ 0 
\end{pmatrix}, \quad r^2 = x_1^2 + x_2^2, \quad 0 = 0_{\R^{N-2}},
\end{equation}
we see that $u$ satisfies
$$
-\Delta u + \frac{u}{r^2} + \lambda u = f(u), \quad x \in \R^N,
$$
where $f \colon \R \rightarrow \R$ is related to $g$ by
$$
| \mathbf{V} | g(\mathbf{V}) = f(|\mathbf{V}|) \mathbf{V}, \quad \mathbf{V} \in \R^N.
$$
The equivalence, under \eqref{a:std}, also holds for weak solutions, cf. \cite[Theorem 5.1]{BMS} (see also \cite[Theorem 1.1]{Bieganowski} and \cite[Theorem 2.1]{GMS}). Consider $K = 2$ and recall that
$$
\cG = \Set {  \begin{pmatrix}
  g   & 0 \\
  0   & I_{N-2}
\end{pmatrix}  |  g \in \cO(2) }.
$$
We introduce the space
$$
\resizebox{\linewidth}{!}{$
\mathcal{H} \coloneqq \Set{\mathbf{U} \in H^1(\R^N; \R^N) | \mathbf{U} (x) = \frac{u(x)}{r}  \begin{pmatrix}
     -x_2 \\ x_1 \\ 0 
\end{pmatrix}  \mbox{ for some } \cG\mbox{-invariant } u \colon \R^N \rightarrow \R }. $}
$$
We introduce the $L^2$-sphere and the $L^2$-disk in $\mathcal{H}$
\begin{align*}
\cD_{\mathcal{H}} (\rho) &\coloneqq \Set{ \mathbf{U} \in \mathcal{H} | \int_{\R^N} |\mathbf{U}|^2 \, \dx \leq \rho }, \\
\cS_{\mathcal{H}} (\rho) &\coloneqq \Set{ \mathbf{U} \in \mathcal{H} | \int_{\R^N} |\mathbf{U}|^2 \, \dx = \rho }.
\end{align*}
Finally, we define the energy functional $E \colon \mathcal{H} \rightarrow \R$ by
$$
E(\mathbf{U}) \coloneqq \frac12 \int_{\R^N} |\nabla \times \mathbf{U}|^2 \, \dx - \int_{\R^N} G(\mathbf{U}) \, \dx, \quad \mathbf{U} \in \mathcal{H},
$$
where
$$
G(\mathbf{U}) \coloneqq \int_0^1 g(t \mathbf{U}) \cdot \mathbf{U} \, \dt.
$$
If $f$ is odd, then, for $\mathbf{U} \in \cH$, we get $|\mathbf{U}| = |u|$, $|\nabla \times \mathbf{U}|_2^2 = |\nabla u|_2^2 + \int_{\R^N} \frac{u^2}{r^2} \, \dx$, $\int_{\R^N} G(\mathbf{U}) \, \dx = \int_{\R^N} F(u) \, \dx$, and $\int_{\R^N} g(\mathbf{U}) \cdot \mathbf{U} \, \dx = \int_{\R^N} f(u)u \, \dx$. In particular, $E(\mathbf{U}) = J(u)$ and
\begin{equation*}
m_{\mathcal{H}}(\rho) \coloneqq \inf\Set{E(\mathbf{U}) | \mathbf{U} \in \cD_\cH(\rho)} = m(\rho),
\end{equation*}
where $m(\rho)$ is defined in \eqref{eq:m}. Recalling the definition of $\rho_*$ in \eqref{eq:threshold}, from Theorem \ref{t:neg} we obtain the following result.

\begin{Th}\label{t:neg-curl}
Assume $N \ge 3$, $s=1$, (\ref{a:std})-(\ref{a:supq}), $f$ is odd, and $F$ is positive somewhere. Suppose, moreover, that $\rho_*$ satisfies \eqref{eq:etai} and $\rho \in \bigl( \rho_* , (2 C_{N,2_\#} \overline{\eta}_\infty)^{-N/2} \bigr)$. If $(\mathbf{U}_n)_n \subset \cD_{\mathcal{H}}(\rho)$ and $\lim_n E(\mathbf{U}_n) = m(\rho)$, then there exist $\mathbf{U} \in \cS_{\mathcal{H}}(\rho)$ and $\lambda > 0$ such that, up to subsequences and translations\footnote{For a function $\mathbf{U} \colon \R^N \to \R^N$ satisfying \eqref{eq:form}, by a translation in the variable $z \in \R^{N-2}$ we mean a translation of $u$.} in the variable $z \in \R^{N-2}$, $\mathbf{U}_n \to \mathbf{U}$ in $H^1 (\R^N;\R^N)$, $E(\mathbf{U}) = m(\rho) \in (-\infty,0)$, and $(\lambda,\mathbf{U})$ is a solution to \eqref{eq:curlcurl}. Furthermore, every $\widetilde{\mathbf{U}} \in \cD_\cH(\rho)$ such that $E(\widetilde{\mathbf{U}}) = m(\rho)$ belongs to $\cS_{\mathcal{H}}(\rho)$ and satisfies \eqref{eq:curlcurl} for some $\widetilde{\lambda} > 0$.
\end{Th}

Observe that Theorem \ref{t:neg-curl} substantially improves \cite[Theorem 8.1.1]{SchinoPHD}. Next, as in \eqref{eq:setB}, we introduce
$$
\cB_{\mathcal{H}}(\rho) \coloneqq \Set{ \mathbf{U} \in \cD_{\mathcal{H}}(\rho) | \left| \nabla \times \mathbf{U} \right|_2^2 > \gamma_* },
$$
where $\gamma_*>0$ is given in Lemma \ref{l:gamma} for $\rho = \rho_*$, $\mu = 1$, and $K=2$. Then, we introduce
\begin{equation*}
m_{*,\mathcal{H}}(\rho) \coloneqq \inf \Set{ E(\mathbf{U}) | \mathbf{U} \in \cB_{\mathcal{H}}(\rho) } = m_{*}(\rho),
\end{equation*}
where $m_*(\rho)$ is defined in \eqref{eq:m_star}. In the setting of the curl-curl problem, Theorem \ref{t:pos} reads then as follows.

\begin{Th}\label{t:pos-curl}
Assume $N \ge 3$, $s=1$, (\ref{a:std})-(\ref{a:supq}), $f$ is odd, $F$ is positive somewhere, $\rho_*$ satisfies \eqref{eq:zero_rho_star} and \eqref{eq:two_etas}, and let $\overline{\rho}_*$ be given by Lemma \ref{lem:A}. If $\rho \in (\overline{\rho}_* , \rho_*]$, $(\mathbf{U}_n)_n \subset \cB_{\mathcal{H}}(\rho)$, and $\lim_n E(\mathbf{U}_n) = m_*(\rho)$, then there exists $\mathbf{U} \in \cB_{\mathcal{H}}(\rho) \cap \cS_{\mathcal{H}}(\rho)$ and $\lambda > 0$ such that, up to subsequences and translations in the variable $z \in \R^{N-2}$, $\mathbf{U}_n \to \mathbf{U}$ in $H^1 (\R^N; \R^N)$ and $(\lambda,\mathbf{U})$ is a solution to \eqref{eq:curlcurl}. In particular, $m_*(\rho) = E(\mathbf{U}) > 0$ if $\rho \in (\overline{\rho}_* , \rho_*)$ and $m_*(\rho_*) = E(\mathbf{U}) = 0 = m(\rho_*)$.
\end{Th}

In the same way, we obtain the following nonexistence result in the class $\mathcal{H}$.

\begin{Th}\label{t:noex-curl}
Assume $N \ge 3$, $s=1$, (\ref{a:std}), {\color{magenta}$f$ is odd,} and $0 < \rho_* < +\infty$. If $2_\# F(t) < f(t) t$ for all $t \ne 0$ or $f(t) t < 2_\# F(t)$ for all $t \ne 0$, then no nontrivial minimisers of $E|_{\cD_\mathcal{H} (\rho_*)}$ exist.
\end{Th}

\appendix
\section{Remarks on the upper bounds for $\rho_*$}\label{App}
Under (\ref{a:std}) and the assumption that $F$ is positive somewhere, Lemma \ref{lem:neg} provides upper bounds for $\rho_*$, i.e., $\rho_* \le \rho_F$ and, if (\ref{a:subcz}) holds, $\rho_* \le (2 C_{N,2_\#} \underline{\eta}_0)^{-N/(2s)}$. Since (\ref{a:subcz}) is a stronger condition than the existence of a point where $F$ is positive, one would expect the latter upper bound not to exceed the former. Somewhat surprisingly, instead, the opposite scenario occurs, at least when the optimal function in the Gagliardo--Nirenberg-type inequality \eqref{eq:GNsym} is essentially bounded. However, one should observe that such a `worse' bound is more likely to be explicitly computed than the other. Moreover, for the sake of completeness, we recall that Remark \ref{r:fF} provides sufficient conditions for such upper bounds to coincide.

\begin{Prop}
Let $w \in X^\mu_\cG$ be the function given by Theorem \ref{t:GN} with $p = 2_\#$. If (\ref{a:std}) and (\ref{a:subcz}) hold and $w \in L^\infty(\R^N)$, then $\rho_F \le (2 C_{N,2_\#} \underline{\eta}_0)^{-N/(2s)}$.
\end{Prop}
\begin{proof}
We consider the case when $\underline{\eta}_0 < \infty$ (the other is similar). Let $\varepsilon > 0$. There exists $\delta > 0$ such that $F(t) \ge (\underline{\eta}_0 - \varepsilon) |t|^{2_\#}$ if $|t| \le \delta$.
Let $\sigma > 0$. Up to internal and external scaling, we can assume that $|w|_\infty = \delta$ and $|w|_{2_\#}^{2_\#} = \sigma / (\underline{\eta}_0 - \varepsilon)$. Then,
\begin{equation*}
	\int_{\R^N} F(w) \, \dx \ge \left( \underline{\eta}_0 - \varepsilon \right) \int_{\R^N} |w|^{2_\#} \, \dx = \sigma.
\end{equation*}
It follows that
\begin{multline*}
	\inf \Set{ [u]_\mu^{N/s} \left|u\right|_2^2 | u \in X_\cG^\mu \text{ and } \int_{\R^N} F(u) \, \dx \ge \sigma } \le [w]_\mu^{N/s} \left|w\right|_2^2 \\
	= \left( C_{N,2_\#}^{-1} |w|_{2_\#}^{2_\#} \right)^{N/(2s)} = \left( \frac{\sigma}{C_{N,2_\#} (\underline{\eta}_0 - \varepsilon)} \right)^{N/(2s)},
\end{multline*}
whence
\begin{equation*}
	\rho_F \le \left( \frac{1}{2 C_{N,2_\#} (\underline{\eta}_0 - \varepsilon)} \right)^{N/(2s)}. \qedhere
\end{equation*}
\end{proof}

\section*{Acknowledgements}
Bartosz Bieganowski is supported by the National Science Centre, Poland (Grant no. 2022/47/D/ST1/00487). Jacopo Schino is a member of GNAMPA (INdAM) and is supported by the GNAMPA-INdAM Project 2026 (CUP E53C25002010001).

We thank the anonymous reviewer for their valuable comments, which allowed us to increase the value of this manuscript.

\end{document}